\documentclass[a4paper]{article}
\usepackage[utf8]{inputenc} 
\usepackage{amsmath}
\usepackage{amssymb}
\usepackage{euscript}
\usepackage{mathrsfs}
\usepackage{bbm}
\usepackage[amsmath,thmmarks]{ntheorem}
\usepackage[ntheorem]{empheq}
\usepackage{tikz-cd}
\usepackage[colorlinks=true,linkcolor=blue,anchorcolor=blue,citecolor=blue,
     		filecolor=blue,urlcolor=blue]{hyperref}
\usepackage{stmaryrd}
\usepackage{bussproofs}\EnableBpAbbreviations
\usepackage{array}
\usepackage{tabularx}
\usepackage[small]{titlesec}
\usepackage{xspace}
\usepackage[hmarginratio={1:1},textwidth=400pt]{geometry}

\usepackage[nameinlink,noabbrev,capitalize]{cleveref}
\usepackage{enumitem}
\setlist[description]{leftmargin=\parindent,labelindent=\parindent}

\makeatletter
\def\namedlabel#1#2{\begingroup
    #2%
    \def\@currentlabel{#2}%
    \phantomsection\label{#1}\endgroup
}
\makeatother

\usepackage{autonum}

\DeclareMathAlphabet{\mathpzc}{OT1}{pzc}{m}{it}

\newcommand{\sfop}[1]{\operatorname{\mathsf{#1}}}
\newcommand{\bfop}[1]{\operatorname{\mathbf{#1}}}

\newcommand{\sub}{\sfop{sub}}

\newcommand{\angs}[1]{\langle #1\rangle}

\newcommand{\tics}{terms-in-con\-text\xspace}
\newcommand{\fic}{for\-mula-in-con\-text\xspace}
\newcommand{\fics}{for\-mulas-in-con\-text\xspace}

\newcommand\llol[1]{\mathbin{%
    \setbox0\hbox{$\multimapinv$}%
    \rlap{\hbox to \wd0{\hss\hss\hss\raisebox{.8\height}{$\scriptstyle 
    #1$}\hss}}\box0
}}
\newcommand\rlol[1]{\mathbin{%
    \setbox0\hbox{$\multimap$}%
    \rlap{\hbox to \wd0{\hss\hss\hss\raisebox{.8\height}{$\scriptstyle #1\,$}
    \hss}}\box0
}}

\newcommand{\ho}{\sfop{ho}}

\newcommand{\mcc}{\mathcal{C}}
\newcommand{\mcd}{\mathcal{D}}

\newcommand{\mch}{\mathcal{H}}

\newcommand{\mck}{\mathcal{K}}

\newcommand{\mcm}{\mathcal{M}}

\newcommand{\mcw}{\mathcal{W}}

\newcommand{\catset}{{\mathbf{Set}}}

\newcommand{\cateff}{{\mathcal{E}\!\mathit{ff}}}

\newcommand{\pfib}{\mathbf{PFib}}

\newcommand{\catcat}{{\mathbf{Cat}}}

\newcommand{\catb}{\mathbb{B}}
\newcommand{\catc}{\mathbb{C}}
\newcommand{\catd}{\mathbb{D}}

\newcommand{\pcaa}{\mathcal{A}}

\newcommand{\trip}{{\EuScript{P}}}
\newcommand{\triq}{{\EuScript{Q}}}

\newcommand{\csep}{\mathrel|}		% separation of
\newcommand{\msep}{\mathrel|}		% separation symbol in subset
\newcommand{\tripower}{{}\mathfrak{P}}			% power types in the higher order language

\newcommand{\op}{^\mathsf{op}}				% for opposite category
\newcommand{\id}{\mathrm{id}}				% identity map
\newcommand{\qdot}{\,.\,}				% point for quantification
\newcommand{\nameof}[1]{\ulcorner #1\urcorner}

\newcommand{\setof}[2]{\{#1\msep #2\}}

\newcommand{\sem}[1]{{\llbracket #1 \rrbracket}}
\newcommand{\csem}[2]{{\llbracket #2 \rrbracket_{#1}}}
\newcommand{\cterm}[2]{{( #1\csep #2 )}}

\newcommand{\aro}{{{(A,\rho)}}}

\newcommand{\bsi}{{{(B,\sigma)}}}

\newcommand{\ctau}{{{(C,\tau)}}}
\newcommand{\cta}{{{(C,\tau)}}}

\newcommand{\tttc}{tripos-to-topos construction\xspace}

\newcommand{\xto}{\xrightarrow}

\newcommand{\ve}{\varepsilon}
\newcommand{\vf}{\varphi}

\newcommand{\imp}{\Rightarrow}

\newcommand{\sh}{\bfop{Sh}}

\newcommand{\brprod}{\Join}

\newcommand{\adj}{\dashv}
\newcommand{\ent}{\vdash}

\newcommand{\smtimes}{\!\times\!}

\newcommand{\catex}{\mathbf{Ex}}

\newcommand{\catpos}{\mathbf{Pos}}

\newcommand{\adjr}[2]{
\ar@/_6pt/[r]_{#1}
\ar@{}[r]|\top
\ar@{<-}@/^6pt/[r]^{#2}
}
\newcommand{\adjrr}[2]{
\ar@/_6pt/[rr]_{#1}
\ar@{}[rr]|\top
\ar@{<-}@/^6pt/[rr]^{#2}
}
\newcommand{\adjl}[2]{
\ar@/^6pt/[l]^{#1}
\ar@{}[l]|\top
\ar@{<-}@/_6pt/[l]_{#2}
}
\newcommand{\adjd}[2]{
\ar@/_6pt/[d]_{#1}
\ar@{}[d]|\dashv
\ar@{<-}@/^6pt/[d]^{#2}
}
\newcommand{\adju}[2]{
\ar@/^6pt/[u]^{#1}
\ar@{}[u]|\dashv
\ar@{<-}@/_6pt/[u]_{#2}
}
\newcommand{\dadj}[4]{ %,left,right,dir,symbol
\ar@/^6pt/[#3]^{#1}
\ar@{}[#3]|{#4}
\ar@{<-}@/_6pt/[#3]_{#2}
}

\newcommand{\fifa}{\EuScript{A}}
\newcommand{\fifb}{\EuScript{B}}

\newcommand{\fifd}{\EuScript{D}}

\newcommand{\conflict}{\ar@{~}}

\def\signed #1{{\leavevmode\unskip\nobreak\hfil\penalty50\hskip2em
  \hbox{}\nobreak\hfil(#1)%
  \parfillskip=0pt \finalhyphendemerits=0 \endgraf}}

\newsavebox\mybox

\newcommand{\gcons}{\textstyle\int\!}

\newcommand{\mor}{\mathrm{mor}}

\newcommand{\inv}{^{-1}}

\newcommand{\catcpos}{\catc\op\to\catpos}

\newcommand{\image}{\mathsf{im}}

\newcommand{\subs}{\subseteq}

\newcommand{\fa}{\forall}
\newcommand{\ex}{\exists}

\newcommand{\sex}{\;\exists}

\newcommand{\seq}{\;=\;}

\makeatletter
\DeclareRobustCommand\onedot{\futurelet\@let@token\@onedot}
\def\@onedot{\ifx\@let@token.\else.\null\fi\xspace}
\newcommand{\ie}{i.e\onedot}
\newcommand{\eg}{e.g\onedot}
\newcommand{\aka}{a.k.a\onedot}

\newcommand{\wrt}{w.r.t\onedot}

\newcommand{\cf}{cf\onedot}

\newcommand{\sequiv}{\;\equiv\;}

\newcommand{\ax}[1]{\AXC{$#1$}}
\newcommand{\ui}[1]{\UIC{$#1$}}
\newcommand{\bi}[1]{\BIC{$#1$}}
\newcommand{\ti}[1]{\TIC{$#1$}}

\newcommand{\drap}{\DP}

\newcommand{\cots}{,\dots,}
\newcommand{\wots}{\wedge\dots\wedge}

\renewcommand{\iff}{if and only if\xspace}
\newcommand{\cofo}{category of fibrant objects\xspace}
\newcommand{\cofos}{categories of fibrant objects\xspace}

\newcommand{\wecat}{we-category\xspace}
\newcommand{\wecats}{we-categories\xspace}
\newcommand{\optopos}{{}\op\to\catpos}
\newcommand{\cfib}{\mathpzc{F}}
\newcommand{\cweq}{\mathpzc{W}}
\newcommand{\eprime}{$\exists$-prime\xspace}
\newcommand{\ecompletion}{$\exists$-completion\xspace}
\newcommand{\ecompletions}{$\exists$-completions\xspace}
\newcommand{\ssc}{s.s.c\onedot}
\newcommand{\imsl}{indexed meet-semi\-lat\-tice\xspace}
\newcommand{\imsls}{indexed meet-semi\-lat\-tices\xspace}

\newcommand{\fmp}{finite-meet-preserving\xspace}

\renewcommand{\lim}{\mathrm{lim}}

\DeclareFontFamily{U}{min}{}
\DeclareFontShape{U}{min}{m}{n}{<-> udmj30}{}
\newcommand{\scol}{\;:\;}

\newcommand{\qtext}[1]{\quad\text{#1}\quad}
\newcommand{\qqtext}[1]{\qquad\text{#1}\qquad}

\newcommand{\cocyl}{\sfop{cocyl}}

\numberwithin{equation}{section}
\newtheorem{theorem}{Theorem}[section]
\newtheorem{proposition}[theorem]{Proposition}
\newtheorem{lemma}[theorem]{Lemma}

\theorembodyfont{\upshape}
\theoremsymbol{\ensuremath{\diamondsuit}}
\newtheorem{definition}[theorem]{Definition}
\theorembodyfont{\normalfont}

\newtheorem{remark}[theorem]{Remark}

\newtheorem{notation}[theorem]{Notation}

\newtheorem{notationandterminology}[theorem]{Notation and Terminology}
\newtheorem{remarks}[theorem]{Remarks}

\newtheorem{examples}[theorem]{Examples}

\newtheorem{void}[theorem]{$\!\!$}
\newenvironment{blank}[1]%
{\begin{void}\textbf{#1.}}%
{\end{void}}
\theoremsymbol{\ensuremath{_\blacksquare}}
\theoremstyle{nonumberplain}
\theoremheaderfont{\itshape}
\newtheorem{proof}{Proof.}

\qedsymbol{\ensuremath{_\blacksquare}}
\theorembodyfont{\itshape}
\theoremheaderfont{\normalfont\bfseries}

\newenvironment{anumerate}{\begin{enumerate}[label=(\textit{\alph*})]}{\end{enumerate}}
\newenvironment{inumerate}{\begin{enumerate}[label=\normalfont(\roman*)]}{\end{enumerate}}

\newcommand{\cateed}{\mathbf{EED}}
\renewcommand{\iff}{if and only if\xspace}

\renewcommand{\pfib}{proto-fibrant\xspace}
\newcommand{\canp}{\catc\langle\trip\rangle}

\newcommand{\bana}{\catb\langle\fifa\rangle}
\newcommand{\canphi}{\catc\langle\Phi\rangle}
\newcommand{\canq}{\catc\langle\triq\rangle}
\newcommand{\csqp}{\catc[\trip]}
\newcommand{\csqq}{\catc[\triq]}
\renewcommand{\catcpos}{\catc\op\to\catpos}
\newcommand{\catbpos}{\catb\op\to\catpos}

\newcommand{\catcp}{\catc[\trip]}

\newcommand{\derline}[2]{{#1}$\quad\implies\quad$ #2}
\newcommand{\tripcoppos}{\trip:\catcpos}
\newcommand{\sort}[1]{{\mathfrak{s}(#1)}}
\newcommand{\cjudg}[3]{{#2\ent_{#1}#3}}
\newcommand{\blakk}[1]{stu}
\renewcommand{\fam}{\mathsf{fam}}
\newcommand{\low}{\mathsf{low}}
\newcommand{\subslice}[3]{{#1}/\!_{#2}{#3}}
\newcommand{\setpos}{\catset\op\to\catpos}
\newcommand{\treff}{{\sfop{eff}}}

\title{Categories of partial equivalence relations as localizations}
\author{Jonas Frey}

\begin{document}

\maketitle

\begin{abstract}
We construct a \emph{category of fibrant objects} $\canp$ in the sense of
K.~Brown from any \emph{indexed frame} (a kind of indexed poset generalizing
triposes) $\trip$, and show that its homotopy category is the Barr-exact
category $\csqp$ of partial equivalence relations and compatible functional
relations. In particular this gives a presentation of \emph{realizability
toposes} as homotopy categories.

We give criteria for the existence of left and right derived functors to
functors $\catc\langle\Phi\rangle:\canp\to\canq$ induced by
finite-meet-preserving transformations $\Phi:\trip\to\triq$ between indexed
frames.
\end{abstract}

\section*{Introduction}

Hyland, Johnstone and Pitts introduced the notion of \emph{tripos}, together
with the \emph{\tttc}, as a powerful tool to describe elementary toposes, in
particular the novel family of \emph{realizability toposes}~\cite{hjp80}. A
tripos is an indexed poset $\trip:\catcpos$ with sufficient structure to
interpret intuitionistic higher order logic, and the associated topos
$\catc[\trip]$ is defined as a category of partial equivalence relations and
compatible functional relations in the internal language of~$\trip$. The proof
that $\catc[\trip]$ is a topos~\cite[Theorem~2.13]{hjp80} relies on the
higher-order structure of $\trip$, but the construction of $\catc[\trip]$ is
well-defined and produces an \emph{exact category}\footnote{I.e.\ a regular
category with effective equivalence relations; called \emph{effective regular}
category in~\cite{elephant1}.} as soon as $\trip$ models the
$(\ex,\wedge,\top)$-fragment of first order logic. In the terminology
of~\cite{maietti2012elementary,maietti2012unifying,maietti2017triposes} such a
$\trip$ is called an \emph{existential doctrine}. If $\trip$ furthermore models
\emph{equality}\,---\,thus the $(\ex,\wedge,\top,=)$-fragment of first order
logic known as \emph{regular logic}~\cite[D1.1]{elephant2}\,---\,then it is
called an \emph{elementary} existential doctrine\footnote{The term `elementary
existential doctrine' was introduced by Lawvere~\cite{lawvere1970equality} for a
slightly different\,---\,in particular non-posetal\,---\,notion, and was adapted
by Maietti and Rosolini to the posetal setting.}. In~\cite{maietti2012unifying},
Maietti and Rosolini show that for elementary existential doctrines $\trip$, the
construction of $\csqp$ can be characterized as left biadjoint to a forgetful
functor $U:\catex\to\cateed$ sending exact categories to their indexed poset of
subobjects, see Example
\ref{ex:indexed-posets}\ref{ex:indexed-posets:canonical-indexing}.

This characterization implies immediately that the assignment sending elementary
existential doctrines $\trip$ to exact categories $\catc[\trip]$ is functorial
\wrt the morphisms of $\cateed$, which are indexed monotone maps preserving the
connectives of regular logic. However, remarkably, it turns out that between
indexed posets with \emph{more}  logical structure (specifically triposes) it is
possible to construct functors from indexed monotone maps $\Phi:\trip\to\triq$
preserving \emph{less} logical structure (specifically only finite meets); a
phenomenon that was exploited in~\cite{hjp80} to construct geometric morphisms
between toposes from suitable adjunctions of indexed monotone maps between
triposes, and has been analyzed in a 2-categorical framework
in~\cite{frey2015triposes}.

\medskip
 
In the present work we show that for a subclass of elementary existential
doctrines which we call \emph{indexed frames}, and which satisfy a stronger
Beck--Chevalley condition, the category $\catcp$ of partial equivalence
relations $\trip:\catcpos$  can be presented as localization of a \emph{\cofo}
$\canp$ (which is homotopically trivial, see Remark~\ref{rem:degenerate}),
and that the functors arising from \fmp transformations $\Phi:\trip\to\triq$
between triposes mentioned above can be understood as \emph{right derived
functors} of functors $\canphi:\canp\to\canq$ between categories of fibrant
objects~(Theorem~\ref{thm:derived}\ref{thm:derived-right}). We also show that
$\canphi$ admits a \emph{left} derived functor whenever $\trip$ has `enough
\eprime predicates' (Theorem~\ref{thm:derived}\ref{thm:derived-left}), and to
set this up we develop  in Section~\ref{sec:eprime} the notion of \emph{\eprime
predicate} and its relation to indexed frames that are obtained by `freely
adding' existential quantification to more primitive indexed posets. 

It should be pointed out that the fibrations in the category-of-fibrant-objects
structure on $\canp$ play a somewhat curious role in our analysis: the fact
that $\catcp$ is the homotopy category can be proven directly without ever
mentioning them, and to prove existence of right derived functors we use a
fibrant replacement-style argument (Lemma~\ref{lem:qfib-kan}) using the
alternative notion of \emph{proto-fibrant
object}~(Definition~\ref{def:proto-fibrant}), ordinary fibrant replacement being
trivial and unhelpful in a setting where everything is already fibrant. However,
the fibrations play an indirect role in the existence proof of left derived
functors (via the notion of \emph{cofibrant object} in a \cofo), and are used in
Remark~\ref{rem:degenerate} to show that the localizations are homotopically
trivial.

\medskip

\subsection*{Related work and acknowledgments.}

This work was inspired by a talk by Jaap van~Oosten about the work presented
in~\cite{vanoosten2015notion}, and influenced by a talk by Benno van den Berg on
the contents of~\cite{vandenberg2016exact}. Presentations of realizability
toposes as homotopy categories have also been given
in~\cite{rosolini2016category,vandenberg2020univalent}.

Thanks to Zhen Lin Low, Rasmus Møgelberg, Mathieu Anel, and especially to Benno
van den Berg for discussions related to the content of this work. Finally, I
want to thank the anonymous referees for their careful reading and very helpful
feedback, which resulted in significant improvements in the article.

\medskip

This material is based upon work supported by the Air Force Office of Scientific
Research under award numbers %
FA9550-15-1-0053 % MURI I
and FA9550-20-1-0305, % TOPOS
and by the U. S. Army Research Office under grant number %
W911NF-21-1-0121. % COHESION

\section{Indexed posets}

An \emph{indexed poset} is a functor $\fifa:\catbpos$ from the opposite of an
arbitrary category $\catb$\,---\,called \emph{base category}\,---\,to the
category $\catpos$ of posets and monotone maps\footnote{Different assumptions on
the relative sizes of $\catb$ and the posets in $\catpos$ are possible here, we
view as basic the case where $\catb$ is small and $\catpos$ is the category
of small posets. Then \eg the case of small posets over locally small categories
can be recovered by postulating a universe and adding assumptions.}. For
$I\in\catb$, the poset $\fifa(I)$ is called the \emph{fiber} of $\fifa$ over $I$
and its elements are called \emph{predicates} on $I$. The monotone maps
$\trip(f):\trip(I)\to\trip(J)$ for $f:J\to I$ are called \emph{reindexing maps}
and are abbreviated $f^*$ when $\trip$ is clear from the context. The
\emph{total category} (\aka \emph{Grothendieck construction}) of $\fifa$ is the
category $\gcons\fifa$ whose objects are pairs $(I\in\catb,\varphi\in\fifa(I))$
and whose morphisms from $(I,\varphi)$ to $(J,\psi)$ are morphisms $f:I\to J$ in
$\catb$ satisfying $\varphi\leq f^*\psi$. The total category admits a forgetful
functor $\gcons\fifa\to\catb$ which is a Grothendieck fibration.

Given indexed posets $\fifa,\fifb:\catbpos$, an \emph{indexed monotone map}
is a natural transformation $\Phi:\fifa\to\fifb$.

An \emph{indexed meet-semilattice} is an indexed poset $\fifa:\catbpos$ where
all fibers are meet-semilattices (which we always assume to be `bounded', \ie to
have a greatest element $\top$) and all reindexing maps preserve finite
meets.
An indexed monotone map $\Phi:\fifa\to\fifb:\catbpos$ between \imsls is called
\emph{cartesian} if it preserves fiberwise finite meets. 

An \emph{indexed frame} is an \imsl $\trip:\catcpos$ on a finite-limit
category $\catc$, satisfying the following three conditions. 
\begin{description}
    \item[\namedlabel{ax:ex}{(Ex)}] 
All reindexing maps $f^*$ (for $f:J\to I$) have left adjoints
$\exists_f:\trip(J)\to\trip(I)$.

    \item[\namedlabel{ax:fr}{(Fr)}] 
    We have
    $
    (\exists_f\varphi)\wedge\psi=\exists_f(\varphi\wedge f^*\psi)
$
    for all $f:J\to I$, $\varphi\in\trip(J)$, $\psi\in\trip(I)$.

    \item[\namedlabel{ax:bc}{(BC)}]
    We have 
    $\exists_h\circ k^*= f^* \circ\exists_g$ for all pullback squares
$
\begin{tikzcd}[sep = small]
            D\ar[r,"k"']\ar[d,"h"'] & B\ar[d,"g"]\\
            A\ar[r,"f"] & C
\end{tikzcd}$ in $\catc$.
\end{description}
Condition \ref{ax:fr} is known as the \emph{Frobenius law}, and \ref{ax:bc}
as the \emph{Beck--Chevalley} condition.

\medskip

A \emph{morphism of indexed frames} is an indexed monotone map that is cartesian
and commutes with $\ex$ along all maps in the base.

\begin{remark}\label{rem:similar-notions} Notions similar to indexed frames
appear under various names in the literature\,---\,compare e.g.\ to
\emph{regular fibrations} in \cite[Definition~4.2.1]{jacobs2001categorical} and
the already mentioned \emph{elementary existential doctrines}, which
only require finite \emph{products} in the base category and use a weaker 
version of the Beck--Chevalley condition.

Stekelenburg's \emph{fibered locales}~\cite{stekelenburg2013realizability}
require \ref{ax:bc} for all existing pullback squares, but the definition is
stated for arbitrary base categories.

These weaker notions are not adequate for the present paper, since the proof of
Theorem~\ref{thm:cofo} relies\,---\,via Lemma~\ref{lem:pullback-rule}\,---\,on
the fact that arbitrary pullbacks exist in the base category and
satisfy~\ref{ax:bc}.
\end{remark}

  A \emph{tripos} is an indexed frame $\tripcoppos$ satisfying moreover the
  following three conditions.
  \begin{description}
    \item[\namedlabel{ax:ha}{(HA)}] All fibers $\trip(I)$ are {Heyting algebras}.
    \item[\namedlabel{ax:u}{(U)}] Besides left adjoints $\ex_f$, all reindexing maps
  $f^*$ have right adjoints $\fa_f$.
    \item[\namedlabel{ax:po}{(PO)}] For every object $I\in\catc$ there exists a
    \emph{power object}, \ie an object $\tripower(I)\in\catc$ together with a
    predicate $\ve_I\in\trip(I\times\tripower(I))$ such that for all $J\in\catc$
    and $\varphi\in\trip(I\times J)$ there exists a (not necessarily unique)
    morphism
    $\nameof{\varphi}:J\to\tripower(I)$ satisfying
    $
    (I\times\nameof{\varphi})^*(\ve_I) = \varphi
    $.
  \end{description}

\begin{remarks}
\begin{anumerate}
\item Whereas we were careful to distinguish between elementary existential
doctrines and indexed frames because of different assumptions on the base
category and phrasings of the Beck--Chevalley condition, there is some ambiguity
around these issues for the notion of tripos, and different variants of the
definition exist in the literature. The above version is given by Pitts
in~\cite{pitts81}.
\item It follows from the presence of the right adjoint and from~\ref{ax:fr}
that the reindexing maps in a tripos preserve the Heyting algebra structure of
the fibers. 

Moreover, the Beck--Chevalley condition for $\fa$ follows from \ref{ax:bc} for
$\ex$ by an argument involving 2-categorical \emph{mates}, because the class of
squares for which we postulate \ref{ax:bc} is closed under transpose (which is
not the case \eg for existential doctrines). 
\end{anumerate}

\end{remarks}

\begin{examples}\label{ex:indexed-posets}
\begin{anumerate}
\item \label{ex:indexed-posets:canonical-indexing} The \emph{canonical indexing}
\begin{equation}
\fam(P):\catset\optopos
\end{equation}
 of a poset $P$ is the $\catset$-indexed poset
defined by $\fam(P)(I)=P^I$ with the pointwise ordering. Monotone maps
$f:P\to Q$ are in bijective correspondence with indexed monotone maps
$\fam(f):\fam(P)\to\fam(Q)$ between canonical indexings. The indexed
poset $\fam(P)$ is an indexed meet-semilattice \iff $P$ is a meet-semilattice,
and it is a tripos \iff it is an indexed frame \iff $P$ is a frame. Provided
domain and codomain have the appropriate structure, $\fam(f)$ is cartesian
\iff $f$ preserves finite meets, and $\fam(f)$ is a morphism of indexed frames
\iff $f$ is a frame morphism.
\item \label{ex:indexed-posets:sub} For every finite-limit category $\catc$ 
there is an indexed meet-semilattice 
\begin{equation}
\sub(\catc):\catc\optopos
\end{equation}
where for each $A\in\catc$, $\sub(\catc)(A)$ is the poset reflection of the full
subcategory of $\catc/A$ on monomorphism, and reindexing is given by pullbacks.

The category $\catc$ is a \emph{regular category}~\cite[A1.3]{elephant1} \iff
$\sub(\catc)$ is an indexed frame.
\end{anumerate}
\end{examples}

\section{The internal language}

The \emph{internal language} of an indexed frame $\tripcoppos$ is a many-sorted
first-order language in the sense of~\cite[Section~D1.1]{elephant2}. It is
generated from a signature whose sorts are the objects of $\catc$, whose
function symbols of arity $A_1,\dots ,A_n\to B$ are the morphisms of type
$A_1\times\dots\times A_n\to B$ in $\catc$, and whose relation symbols of arity
$A_1,\dots ,A_n$ are the elements of $\trip(A_1\times\dots\times
A_n)$\footnote{Note that the same predicate in $\trip$ can give rise to
different relation symbols, \eg $\rho\in\trip(A\times B)$ can both be viewed as
a binary relation symbol of arity $A,B$ and as unary relation symbol of
arity $A\times B$. The same is true for morphisms and function
symbols. The intended interpretation will always be clear from the context.}. 

Over this signature we consider \emph{terms}\,---\,which are built up from
sorted variables\footnote{We don't annotate variables with sorts as the sorts
are always clear from the context.} and function symbols, subject to matching
arities and sorts\,---\,and \emph{regular formulas}, which are generated from
\emph{atomic formulas} $\varphi(\vec t)$ (where $\varphi$ is a relation symbol
and $\vec t$ is a list of terms matching its arity) and $s=t$ (where $s$ and $t$
are terms of the same sort), using the connectives of conjunction~$\wedge$,
truth~$\top$, and existential quantification~$\exists$.

A \emph{context} is a list $\vec x$ of distinct variables. We say that a term
$t$ or formula $P$ is \emph{in context} $\vec x$, if all of its free variables
are contained in $\vec x$.

Instead of writing that  a term $t$ or formula $P$ is in context $\vec x$, we
also write $\cterm{\vec x}{t}$ is a \emph{term-in-context} and $\cterm{\vec
x}{P}$ is a \emph{\fic}.

We write $\sort{x}$ and $\sort{t}$ for the sort of a variable and a term,
respectively, and we use the shorthand
$\sort{x_1,\dots,x_n}\seq\sort{x_1}\times\dots\times\sort{x_n}$ for contexts.

The \emph{interpretation} of terms-in-context and formulas-in-context is defined
by structural induction by the clauses in Table~\ref{table:interpretation}.
In general, the interpretation of a term-in-context $(\vec x\csep t)$ is a
morphism $ \sem{t}_{\vec x}: \sort{\vec x} \to\sort{t}$ in $\catc$, and the
interpretation of a formula-in-context $(\vec x\csep P)$ is a predicate
$\sem{P}_{\vec x} \in \trip(\sort{\vec x })$.
\begin{table}
\centering{\fbox{
\begin{minipage}{0.77\linewidth}
\begin{align}
\sem{x_i}_{\vec x} &= \pi_i
\\
\sem{f(t_1,\dots,t_n)}_{\vec x}
&= f\circ\langle\sem{t_1}_{\vec x},\dots,\sem{t_n}_{\vec x}\rangle
\\
\sem{\varphi(t_1,\dots,t_n)}_{\vec x} 
&= \langle\sem{t_1}_{\vec x},
\dots,\sem{t_n}_{\vec x}\rangle^*(\varphi)
\\
\sem{t = u}_{\vec x} &= 
\langle \sem{t}_{\vec x}, \sem{u}_{\vec x}\rangle^*
(\exists_{\delta}\top
)
\\
\sem{P\wedge Q}_{\vec x} &= \sem{P}_{\vec x}
\wedge\sem{Q}_{\vec x}
\\
\sem{\top}_{\vec x} &= \top
\\
\sem{\exists y\qdot P}_{\vec x} &=
\exists_\pi(\sem{P}_{{\vec x},y})
\end{align}
In the fourth clause $\delta$ is the diagonal map
$\sort{t}\to\sort{t}\smtimes\sort{t}$, and in the last clause $\pi$ is the
projection
$\sort{\vec x}\smtimes\sort{y}\to\sort{\vec x}$.
\end{minipage}}}
\caption{Interpretation of the internal language}\label{table:interpretation}
\end{table}

When defining a predicate in an indexed frame by a formula in the internal 
language, we normally write
$\varphi(\vec x)\equiv P
$ instead of $
\varphi = \csem{\vec x}{P}
$.

The following standard lemmas are
verified by structural induction.
\begin{lemma}[Weakening]\label{lem:weak}
We have 
\begin{itemize}
  \item $\sem{t}_{{\vec x},y,\vec z} = \sem{t}_{{\vec x},{\vec z}}\circ\pi$
  \item $\sem{P}_{{\vec x},y,{\vec z}} =
  \pi^*(\sem{P}_{{\vec x},{\vec z}})$
\end{itemize}
for all terms-in-context $(\vec x,\vec z\csep t)$ and formulas-in-context $(\vec
x,\vec z\csep P)$, where $\pi:\sort{\vec x,y,\vec z}\to\sort{\vec x,\vec
z}$ is the obvious projection. \qed
\end{lemma}
\begin{lemma}[Substitution] \label{lem:subs}
We have
\begin{itemize}
  \item $\sem{t[u/y]}_{\vec x}=\sem{t}_{{\vec x},y}\circ\langle\id_\sort{{\vec
  x}},\sem{u}_{{\vec x}}\rangle$
  \item
$\sem{P[u/y]}_{\vec x}=\langle\id_\sort{{\vec x}},\sem{u}_{{\vec
x}}\rangle^*(\sem{P}_{{\vec x},y})$ \end{itemize} for all formulas-in-context
$(\vec x,y\csep P)$ and terms-in-context $(\vec x,y\csep t)$, $(\vec
x\csep u)$ such that $\sort{y}=\sort{u}$.
\qed
\end{lemma}
Terms-in-context $\cterm{\vec x}{t}$ and $\cterm{\vec x}{u}$ (or \fics
$\cterm{\vec x}{P}$ and $\cterm{\vec x}{Q}$) are called \emph{semantically
equal}, if $\csem{\vec x}{t}=\csem{\vec x}{u}$ (or $\csem{\vec x}{P}=\csem{\vec
x}{Q}$).

The following lemma justifies \emph{local rewriting}.  Because of the presence
of quantifiers it cannot be deduced from the substitution lemma, but it can be
proven straightforwardly by structural induction.
\begin{lemma}[Congruence]\label{lem:cong}
Semantic equality of terms and formulas in-con\-text
is 
preserved when
replacing sub\tics (or
sub\fics) of a \fic $\cterm{\vec x}{P}$ by semantically equal ones.\qed
\end{lemma}

A \emph{judgment} in the internal language is an expression of the form
$\Gamma\ent_{\vec x}Q$, where $\Gamma \equiv P_1,\dots,P_n$
is a list of formulas
in context ${\vec x}$, and $Q$ is a formula in context $\vec x$.
We say that the judgment is \emph{valid} (or \emph{holds}), if 
\begin{equation}
  \sem{P_1}_{\vec x}\wedge\dots\wedge\sem{P_n}_{\vec x}
  \leq \sem{Q}_{\vec x}
  \qtext{in} \trip(\sort{{\vec x}})\;.
\end{equation}
For convenience introduce the abbreviation 
$\sem{\Gamma}_{\vec x}=\sem{P_1}_{\vec x}\wedge\dots\wedge\sem{P_n}_{\vec x}$
for the left-hand side of this inequality.

A \emph{deduction} is an $(n+1)$-tuple of judgments $(\mch_1,\dots,\mch_n;\mck)$
for some $n\geq 0$, where the $\mch_i$ are called \emph{hypotheses} and $\mck$
the \emph{conclusion}. The deduction is called \emph{admissible} if the validity
of the hypotheses implies the validity of the conclusion. As is customary, we
often write deductions with a horizontal line as 
\begin{equation}
\ax{\mch_1\quad\dots\quad\mch_n}\ui{\mck}\drap\;\;.
\end{equation}

A \emph{rule} is a `deduction scheme', \ie a deduction containing placeholders
for terms and formulas and contexts. A rule is called admissible if all
syntactically-correct instantiations of these placeholders yield admissible
deductions.

The following \emph{soundness theorem} is straightforward and holds in fact for
the internal language of any elementary existential doctrine.
\begin{theorem}[Soundness]\label{thm:sound} The rules of regular logic in
Table~\ref{table:reg-rules} are admissible. \hfill$_\blacksquare$
\end{theorem}
\begin{table}
\centering{\fbox{
\begin{minipage}{0.8\linewidth}
\begin{tabular}{ccc}\\[-2mm]
\ax{\phantom{|}}\ui{P_1,\dots,P_n\ent_{\vec x}P_i}\drap &
\ax{\Gamma\ent_{\vec x,y} R[t/y]}\ui{\Gamma\ent_{\vec x}\ex y\qdot
 R}\drap &
\ax{\Gamma\ent_{\vec x}\ex y\qdot  R} \ax{\Gamma, R\ent_{\vec x,y}P}
\bi{\Gamma\ent_{\vec x}P}\drap
\\[5mm]
\ax{\phantom{|}}\ui{\Gamma\ent_{\vec x}\top}\drap 
&
\ax{\phantom{|}}\ui{\Gamma\ent_{\vec x}t=t}\drap 
&
\ax{\Gamma\ent_{\vec x} R[s/y]}
\ax{\Gamma\ent_{\vec x}s=t}
\bi{\Gamma\ent_{\vec x} R[t/y]}\drap
\\[5mm]
\ax{\Gamma\ent_{\vec x}P\wedge Q}
\ui{\Gamma\ent_{\vec x}P}\drap
&
\ax{\Gamma\ent_{\vec x}P\wedge Q}
\ui{\Gamma\ent_{\vec x} Q}\drap
& 
\ax{\Gamma\ent_{\vec x}P} \ax{\Gamma\ent_{\vec x} Q}
\bi{\Gamma\ent_{\vec x}P\wedge Q}\drap
\end{tabular}\\[3mm]
\end{minipage}
}}
\caption{The rules of regular logic}\label{table:reg-rules}
\end{table}

On the other hand, the next lemma relies on the stronger version of the 
Beck--Chevalley condition postulated in indexed frames.
\begin{lemma}\label{lem:pullback-rule}
Assume that 
\begin{equation}\label{eq:pullback}
\begin{tikzcd}
    D   \ar[r,"k"']\ar[d,"h"']
&   B   \ar[d,"g"]
\\  A   \ar[r,"f"]
&   C
\end{tikzcd}
\end{equation}
is a pullback square in $\catc$, and $\Gamma$ and $ Q$ are respectively a list
of formulas and a formula in the same context $\vec x,y,z$, with $\sort{y} = A$
and $\sort{z} = B$.
Then the deduction 
\begin{equation}\label{eq:ded}
\ax{\Gamma[hp/y,kp/z]\ent_{\vec x,p} Q[hp/y,kp/z]}
\ui{\Gamma,fy=gz\ent_{\vec x,y,z} Q}
\drap
\end{equation}
is admissible.
\end{lemma}
\begin{proof}
Let $X=\sort{\vec x}$. The assumption that~\eqref{eq:pullback} is a pullback
implies that the square
\begin{equation}\begin{tikzcd}
    X\times D
        \ar[r,"v"']
        \ar[d,"m"']
&	C
        \ar[d,"\delta"]
\\	X\times A \times B
        \ar[r,"u"]
&	C\times C
\end{tikzcd}
\qquad
\begin{matrix}
m = X\times\angs{h,k}\\
u = \angs{f\circ\pi_2,g\circ\pi_3}\\
v = g\circ k \circ \pi_2 = f\circ h \circ \pi_2
\end{matrix}
\end{equation}
is a pullback as well. The hypothesis of the deduction~\eqref{eq:ded} unfolds
into the inequality 
\begin{equation}\label{eq:hyp}
m^*\csem{\vec x,y,z}{\Gamma}\leq m^*\csem{\vec x,y,z}{Q}.
\end{equation}
For the interpretation of the left-hand side of the conclusion we have
\begin{align}
\csem{\vec x,y,z}{\Gamma}\wedge\csem{\vec x,y,z}{fy=gz}
&=\csem{\vec x,y,z}{\Gamma}\wedge u^*\,\ex_\delta\,\top\\
&=\csem{\vec x,y,z}{\Gamma}\wedge \,\ex_m\,\top&&\text{by \ref{ax:bc}}\\
&=\ex_m\,m^*\,\csem{\vec x,y,z}{\Gamma}&&\text{by \ref{ax:fr}},
\end{align}
and we see that the last term is less-or-equal $\csem{\vec x,y,z}{Q}$ \iff
the inequality \eqref{eq:hyp} holds.
\end{proof}
\begin{blank}{Interpreting non-regular connectives in triposes}
Triposes admit a richer internal language than general indexed frames, since
they can model \emph{all} connectives of intuitionistic first-order logic, not
only the regular fragment. The \emph{non-regular} connectives $\fa,
\vee,\bot,\imp$ are interpreted by the clauses
\begin{align}
\csem{\vec x}{\fa y\qdot P}&= \fa_\pi\csem{\vec x,y}{P}&
\sem{P\vee Q}_{\vec x}&= \sem{P}_{\vec x}\vee\csem{\vec x}{ Q}\\
\sem{P\imp Q}_{\vec x}&= \sem{P}_{\vec x}\imp\csem{\vec x}{ Q}&
\csem{\vec x}{\bot}&=\bot
\end{align}
where again, $\pi$ is the appropriate projection and `$\vee$', `$\bot$', and
`$\imp$' on the right-hand sides denote binary join, least element, and Heyting
implication in the Heyting algebra $\trip(\sort{\vec x})$. The augmented
language satisfies the \emph{Weakening}, \emph{Substitution}, and
\emph{Congruence Lemmas}~\ref{lem:weak}, \ref{lem:subs} and \ref{lem:cong}
without modification, and the \emph{Soundness Theorem}~\ref{thm:sound} for the
set of rules in Table~\ref{table:reg-rules} and the following rules for the 
new connectives.
\vspace{-10pt} 
\begin{equation}
\setlength\extrarowheight{14pt}
\setlength{\tabcolsep}{12pt}
\begin{tabular}{ccc}
\ax{\cjudg{\vec x,y}{\Gamma}{ R}}
\ui{\cjudg{\vec x}{\Gamma}{\fa y\qdot R}}
\drap
&
\ax{\cjudg{\vec x}{\Gamma,P}{ Q}}
\ui{\cjudg{\vec x}{\Gamma}{P\imp Q}}
\drap
&
\ax{\cjudg{\vec x}{\Gamma}{P}}
\ui{\cjudg{\vec x}{\Gamma}{P\vee Q}}
\drap
\\
\ax{\cjudg{\vec x}{\Gamma}{\fa y\qdot R}}
\ui{\cjudg{\vec x}{\Gamma}{ R[t/y]}}
\drap
&
\ax{\cjudg{\vec x}{\Gamma}{P\imp Q}}
\ax{\cjudg{\vec x}{\Gamma}{P}}
\bi{\cjudg{\vec x}{\Gamma}{ Q}}
\drap
&
\ax{\cjudg{\vec x}{\Gamma}{ Q}}
\ui{\cjudg{\vec x}{\Gamma}{P\vee Q}}
\drap
\\
\ax{\cjudg{\vec x}{\Gamma}{\bot}}
\ui{\cjudg{\vec x}{\Gamma}{P}}
\drap
&
\multicolumn{2}{c}{
\ax{\cjudg{\vec x}{\Gamma}{P\vee Q}}
\ax{\cjudg{\vec x}{\Gamma,P}{S}}
\ax{\cjudg{\vec x}{\Gamma, Q}{S}}
\ti{\cjudg{\vec x}{\Gamma}{S}}
\drap
}
\end{tabular}
\end{equation}
Finally, to account for the power objects in triposes we could extend the
language by a higher order term former for subset comprehension, or
alternatively express comprehension as an axiom
scheme~\cite[Axiom~4.1]{pitts2002}. We don't need this here\,---\,our only use of
internal language for triposes is in the proof of
Theorem~\ref{thm:enough-fib-cofib} and there the first order language is
sufficient.
\end{blank}

\section{\hspace{-2pt}The category of partial equivalence relations and
compatible maps}

\begin{definition}\label{def:fct} Let $\fifa:\catb\op\to\catpos$ be an indexed
meet-semilattice on a finite-product category. The category $\catb\angs{\fifa}$
is defined as follows.
\begin{itemize}
 \item Objects are pairs $(A\in\catb,\rho\in\fifa(A\times A))$ such that 
 the judgments
 \begin{description}
  \item[\namedlabel{ax:sym}{(sym)}]
$\rho(x,y)\ent_{x,y} \rho(y,x)$
  \item[\namedlabel{ax:trans}{(trans)}] $\rho(x,y),\rho(y,z)\ent_{x,y,z}\rho(x,z)$
\end{description}
hold in $\fifa$.
 \item Morphisms from $(A,\rho)$ to $(B,\sigma)$ are maps $f:A\to B$ in $\catb$
 such that
 \begin{description}
 \item[\namedlabel{ax:compat}{(compat)}]
$\rho(x,y)\ent_{x,y}\sigma(fx,fy)$
\end{description}
holds in $\fifa$.
 \item Composition and identities are inherited from $\catb$.
\end{itemize}
\end{definition}
 
\begin{notationandterminology}
\begin{anumerate}
\item 
If $\rho\in\fifa(A\times A)$ satisfies \ref{ax:sym} and {\ref{ax:trans}}, we 
call it a \emph{partial equivalence relation}. If $f:A\to B$ satisfies \ref{ax:compat}, we say that it is \emph{compatible} with $\rho$ and $\sigma$.
\item
We  write $\rho_0 :=\delta^*_A\rho$ for the reindexing of a partial equivalence
relation along the diagonal $\delta_A:A\to A\times A$, and call it the 
\emph{support} of $\rho$.
\item
In the internal language we simply write $\rho x$ instead of $\rho_0x$ or
$\rho(x,x)$.
\end{anumerate}
\end{notationandterminology}

\begin{lemma} \label{lem:fct-props} Let $\fifa:\catcpos$ be an indexed
meet-semilattice on a finite-limit category $\catc$.
\begin{inumerate}
\item The forgetful functor
$U : \catc\angs{\fifa}\to\catc$ 
has a right adjoint $\nabla$.
\item The category $\catc\angs{\fifa}$ {has} finite limits and $U$ preserves them.
\item\label{lem:fct-props-iso} A morphism $f:\aro\to\bsi$ is an isomorphism in
$\catc\angs{\fifa}$ \iff $f$ is an isomorphism in $\catc$ and $(f\times
f)^*\sigma=\rho$.
\end{inumerate}
\end{lemma}

\begin{proof}
The right adjoint is given by $\nabla(A)=(A,\top)$.
The terminal object in $\catc\angs{\fifa}$ is $(1,\top)$. A pullback of 
$\aro\xrightarrow{f}\cta\xleftarrow{g}\bsi$ is given by
\begin{equation}
\begin{tikzcd}
    (D,\rho\brprod_C\sigma)
        \ar[r,"k"']
        \ar[d,"h"']
&   \bsi
        \ar[d,"g"]
\\  \aro
        \ar[r,"f"]
&   \cta
\end{tikzcd}
\vspace{-10pt}
\end{equation}
where 
$
\begin{tikzcd}[sep = small]
 D\ar[r,"k"']\ar[d,"h"'] & B\ar[d,"g"]\\
 A\ar[r,"f"] & C
\end{tikzcd}
$
is a pullback in $\catc$ and
$
 (\rho\brprod_C\sigma)(p,q) \equiv \rho(hp,h q)\wedge\sigma(k p,k q).
$ 
For the third claim, the necessity of the conditions becomes obvious by 
considering an inverse to $f:\aro\to\bsi$. Conversely, the conditions also 
allow to construct this inverse.
\end{proof}
\begin{remark}
Setting $\cta=(1,\top)$ in the construction of the pullback, we get a
representation of binary products as $\aro\times(B,\sigma)=(A\times
B,\rho\brprod\sigma)$ with $(\rho\brprod\sigma)(u,v)\equiv
\rho(\pi_1u,\pi_1v)\wedge\sigma(\pi_2u,\pi_2v)$.
\end{remark}

\begin{definition}\label{def:descent-pred} Given an indexed meet-semilattice
$\fifa:\catbpos$ on a finite-product category and an object $\aro\in\bana$, we
say that $\varphi\in\fifa(A)$ is a \emph{$\rho$-descent predicate} if the
judgments
\begin{equation}
\varphi x\ent_x\rho x\qqtext{and}\varphi x,\rho(x,y)\ent_{x,y}\varphi y
\end{equation}
hold in $\fifa$. 
\end{definition}
\begin{remarks}\label{rem:desc}
\begin{anumerate}
\item\label{rem:desc:ipos} For every indexed meet-semilattice $\fifa:\catbpos$
on a finite-product category, we can define an indexed meet-semilattice
$\widetilde{\fifa}:\bana\op\to\catpos$ where the elements of
$\widetilde{\fifa}\aro$ are the $\rho$-descent predicates.
In~\cite[Theorem~4.8]{maietti2012elementary}, $\widetilde{\fifa}$ is
characterized as a free completion of $\fifa$ with \emph{comprehension} and
\emph{descent quotients}.
\item In~\cite[pg.~67]{vanoosten2008realizability} van Oosten writes `strict 
relation on $\aro$' for what we call $\rho$-descent predicate. 
In~\cite{frey2013fibrational}, the terminology `predicate compatible with 
$\rho$' is used.
\end{anumerate}
\end{remarks}

\section{The category \texorpdfstring{$\canp$}{C<P>} as a category of fibrant
objects}

In this section we show that for every indexed frame $\tripcoppos$, the category
$\canp$ can be equipped with the structure of a `category of fibrant objects' in
a natural way. We start by recalling the definition
from~\cite{brown1973abstract}.
\begin{definition}
 A \emph{category of fibrant objects} is a category $\mcc$ with finite 
 products, together with two distinguished
 classes $\cfib,\cweq\subs\mor(\mcc)$ of morphisms whose elements are called
 \emph{fibrations} and \emph{weak equivalences}, respectively. Morphisms in 
 $\cfib\cap\cweq$ are called \emph{trivial fibrations}. These classes are 
 subject  to the following axioms.
 \begin{description}
  \item[\namedlabel{ax:a}{(A)}]
  All isomorphisms are weak equivalences, and for any composable pair
  $A\xrightarrow{f}B\xrightarrow{g}C$, if either two of the three morphisms $f$,
  $g$, and $gf$ are weak equivalences, then so is the third.
  \item[\namedlabel{ax:b}{(B)}]
  The class of fibrations contains all isomorphisms and is closed under composition.
  \item[\namedlabel{ax:c}{(C)}] 
  Pullbacks of fibrations along arbitrary maps exist and are fibrations.
  Pullbacks of trivial fibrations are trivial fibrations.
  \item[\namedlabel{ax:d}{(D)}] 
  For any $X\in\mcc$ there exists a \emph{path object}, i.e.\ a 
  factorization 
  \begin{equation}
   X\xrightarrow{s} P(X) \xrightarrow{d=\langle d_0,d_1\rangle} X\times X
  \end{equation}
  of the diagonal with $s\in\cweq$ and $d\in\cfib$.
  \item[\namedlabel{ax:e}{(E)}] 
  For any $X\in\mcc$, the map $X\to 1$ is a fibration.
 \end{description}
\end{definition}
For the remainder of the section let $\tripcoppos$ be a fixed indexed frame.
\begin{definition}\label{def:fib-equiv}
A morphism $f:\aro\to\bsi$ in $\canp$ is a \emph{fibration} if the judgment
\begin{description}
 \item[\namedlabel{ax:fib}{(fib)}]
$\rho x,\sigma(fx,u)\ent_{x,u}\exists y\qdot \rho(x,y)\wedge fy=u$
\end{description}
holds in $\trip$. It is a \emph{weak equivalence} if the judgments
\begin{description}
 \item[\namedlabel{ax:inj}{(inj)}]
$\rho x,\sigma(fx,fy),\rho y\ent_{x,y}\rho(x,y)$ 
 \item[\namedlabel{ax:esurj}{(esurj)}]
$\sigma u\ent_u\exists x\qdot \rho x\wedge \sigma(fx,u)$
\end{description}
hold in $\trip$.
\end{definition}
\begin{lemma}\label{lem:tfib-eq-inj-surj}
$f:\aro\to\bsi$ is a trivial fibration if and only if \ref{ax:inj} and
\begin{description}
\item[\namedlabel{ax:surj}{(surj)}]
$\sigma u\ent_{u}\exists x\qdot \rho x \wedge f x=u$
\end{description}
hold in $\trip$.
\end{lemma}
\begin{proof}
We show the following implications.
\begin{itemize}
\item 
\ref{ax:surj} $\implies$ \ref{ax:esurj}
\item 
\ref{ax:inj}, \ref{ax:surj} $\implies$ \ref{ax:fib}
\item 
\ref{ax:esurj}, \ref{ax:fib} $\implies$ \ref{ax:surj}
\end{itemize}
The proof is carried out in the internal language. In the following enumerated
lists, each line represents an admissible deduction, with hypotheses and
conclusion separated by the arrow `$\implies$'. In each line, the hypotheses are
known to be valid either because they are assumptions, or because they were
established in the previous lines. The admissibility of the individual
deductions follows from the rules of regular logic (Table~\ref{table:reg-rules})
by elementary logical reasoning.

\medskip\noindent
Implication \ref{ax:inj}, \ref{ax:surj} $\implies$ \ref{ax:fib}:
\begin{enumerate}
\item \label{itm:aaa}
\derline{\ref{ax:sym}, \ref{ax:trans}}
{$\sigma(fx,u)\ent_{x,u}\sigma u$}
\item\label{itm:zza}
\derline{\ref{itm:aaa}, \ref{ax:surj}}
{$\sigma(fx,u)\ent_{x,u}\exists y\qdot \rho y\wedge fy=u$}
\item\label{itm:zze}
\derline{\ref{ax:inj}}{$\rho x ,\sigma(fx,u),\rho y ,fy =u\ent_{x,y,u}
\rho(x,y)$}
\item\label{itm:zzc}
\derline{\ref{itm:zze}%
}
{$\rho x ,\sigma(fx,u),\rho y ,fy =u\ent_{x,y,u}\rho(x,y)\wedge fy=u$}
\item\label{itm:zzb}
\derline{\ref{itm:zzc}}
{$\rho x,\sigma(fx,u),\rho y,fy=u\ent_{x,y,u}\exists y\qdot\rho(x,y)\wedge fy=u$}
\item \derline{\ref{itm:zza}, \ref{itm:zzb}}
{$\rho x ,\sigma(fx,u)\ent_{x,u}\exists y\qdot \rho(x,y)\wedge fy=u$}
\end{enumerate}
Implication \ref{ax:surj} $\implies$ \ref{ax:esurj}:
\begin{enumerate}
 \item \label{itm:wia}
 \derline{\ref{ax:compat}}
 {$\rho x\ent_x\sigma (fx)$}
 \item \label{itm:wib}
 \derline{\ref{itm:wia}}
 {$\rho x, fx=u\ent_{x,u}\sigma(fx,u)$}
 \item \label{itm:wic}
 \derline{\ref{itm:wib}}
 {$\rho x, fx=u\ent_{x,u}\rho x\wedge\sigma(fx,u)$}
\item \label{itm:wid}
 \derline{\ref{itm:wic}}
 {$\exists x\qdot\rho x\wedge fx=u\ent_u\exists x\qdot\rho x\wedge\sigma(fx,u)$}
 \item \label{itm:wie}
 \derline{\ref{itm:wid}, \ref{ax:surj}}
 {$\sigma u\ent_u\exists x\qdot \rho x\wedge \sigma(fx,u)$}
\end{enumerate}
Implication \ref{ax:esurj}, \ref{ax:fib} $\implies$ \ref{ax:surj}:
\begin{enumerate}
 \item \label{itm:jka}
 \derline{\ref{ax:sym}, \ref{ax:trans}}
 {$\rho(x,y)\ent_{x,y,u}\rho y$}
 \item \label{itm:jkb}
 \derline{\ref{itm:jka}}
 {$\rho(x,y),fy=u\ent_{x,y,u}\rho y \wedge fy=u$}
 \item \label{itm:jkc}
 \derline{\ref{itm:jkb}}
 {$\exists y\qdot\rho(x,y)\wedge fy=u\ent_{x,u}\exists y\qdot \rho y \wedge 
 fy=u$}
 \item \label{itm:jkd}
 \derline{\ref{itm:jkc}, \ref{ax:fib}}
 {$\rho x,\sigma(fx,u)\ent_{x,u}\exists y\qdot \rho y \wedge fy=u$}
 \item \label{itm:jke}
 \derline{\ref{itm:jkd}, \ref{ax:esurj}}
 {$\sigma u\ent_{u}\exists y\qdot \rho y \wedge fy=u$}
\end{enumerate}
\end{proof}
\begin{remark}\label{rem:surj-reform}
Stated variable-freely, the condition \ref{ax:surj} reduces to the inequality
$\sigma_0\leq\ex_f\rho_0$, and since the reverse inequality follows
from~\ref{ax:compat} this is equivalent to the equality $\sigma_0=\ex_f\rho_0$.
\end{remark}
\begin{definition}\label{def:restrictper} Given an object $\aro\in\canp$ and a
$\rho$-descent predicate $\varphi\in\trip(A)$, we define the \emph{restriction}
$\rho|_\varphi$ of $\rho$ to $\varphi$ by 
$(\rho|_\varphi)(x,y)\;\equiv\;\rho(x,y)\wedge\varphi(x)$.
\end{definition}
The following is easy to see.
\begin{lemma}\label{lem:restrict-fib}
If $\rho\in\trip(A\times A)$ is a partial equivalence relation and
$\vf\in\trip(A)$ is a $\rho$-descent predicate, then the restriction $\rho|_\vf$
is a partial equivalence relation, and the identity $\id:A\to A$ in $\catc$
induces a monomorphism 
\begin{equation}
  (A,\rho|_\varphi)\to\aro\end{equation}
in $\canp$ which is a fibration. \qed
\end{lemma}
\begin{theorem}\label{thm:cofo}
$\canp$ with the classes of fibrations and weak equivalences 
from Definition~\ref{def:fib-equiv} is a category of fibrant objects.
\end{theorem}
\begin{proof}
Using Lemma~\ref{lem:fct-props}\ref{lem:fct-props-iso} it is easy to see that
\ref{ax:fib}, \ref{ax:inj}, and \ref{ax:esurj} hold for isomorphisms, and are
stable under composition. Given a composable pair
$\aro\xrightarrow{f}\bsi\xrightarrow{g}\cta$, if \ref{ax:inj} holds for $gf$,
then it holds for $f$, and if \ref{ax:esurj} holds for $gf$, then it holds for
$g$; for the same reason that initial segments of injective functions are
injective and end segments of surjective functions are surjective. Furthermore,
it is easy to show that \ref{ax:esurj} for $gf$ and \ref{ax:inj} for $g$ implies
\ref{ax:esurj} for $f$, and that \ref{ax:inj} for $gf$ and \ref{ax:esurj} for
$f$ implies \ref{ax:inj} for $g$, again formalizing set theoretic arguments.
This shows conditions \ref{ax:a} and \ref{ax:b}.

For condition \ref{ax:c} consider a pullback square
\begin{equation}
\begin{tikzcd}
    (D,\rho\brprod_C\sigma)
        \ar[r,"k"']
        \ar[d,"h"']
&   \bsi
        \ar[d,"g"]
\\  \aro
        \ar[r,"f"]
&   \cta
\end{tikzcd}
\end{equation}
and assume that $g$ is a fibration. To show the validity of \ref{ax:fib} for $h$
(abbreviated \ref{ax:fib}$(h)$) we use internal language with the notation
introduced in the proof of Lemma~\ref{lem:tfib-eq-inj-surj}. In the step from
\ref{itm:ttf} to \ref{itm:ttg} we use the admissible rule from
Lemma~\ref{lem:pullback-rule}.
\begin{enumerate}
 \item \label{itm:ttb}
 \derline{\ref{ax:compat}}
 {$\rho (hp,x) \ent_{p,x} \tau (g(kp),fx)$}
 \item \label{itm:ttc}
 \derline{\ref{ax:fib}$(g)$}
 {$\sigma (kp),\tau (g(kp),fx) \ent_{p,x} \exists  v\qdot \sigma (kp,v)\wedge 
gv = fx$}
 \item \label{itm:ttd}
 \derline{\ref{itm:ttb}, \ref{itm:ttc}}
 {$\sigma (kp),\rho (hp,x) \ent_{p,x} \exists  v\qdot \sigma (kp,v)\wedge 
gv = fx$}
 \item \label{itm:tte}
 \derline{}
 {$\rho (hp,hq^*),\sigma (kp,kq^*) \ent_{p,q^*} \rho (hp,hq^*)\wedge \sigma (kp,kq^*)\wedge hq^*=hq^*$}
 \item \label{itm:ttf}
 \derline{\ref{itm:tte}}
 {$\rho (hp,hq^*),\sigma (kp,kq^*) 
\ent_{p,q^*} \exists q\qdot\rho (hp,hq)\wedge\sigma (kp,kq)\wedge hq=hq^*$}
 \item \label{itm:ttg}
 \derline{\ref{itm:ttf}}
 {$\rho (hp,x),\sigma (kp,v),gv = fx \ent_{p,x,v} \exists q\qdot\rho (hp,hq)\wedge\sigma (kp,kq)\wedge hq=x$}
 \item \label{itm:tth}
 \derline{\ref{itm:ttd}, \ref{itm:ttg}}
 {$\sigma (kp),\rho (hp,x) \ent_{p,x} \exists q\qdot\rho (hp,hq)\wedge \sigma (kp,kq),hq=x$}
 \item \label{itm:tti}
 \derline{\ref{itm:tth}}
 {\ref{ax:fib}$(h)$}
 \end{enumerate}
This shows that fibrations are stable under pullback. To show that  
\emph{trivial} fibrations are stable under pullback, we show pullback
stability of conditions \ref{ax:inj} and \ref{ax:surj} separately.

Pullback stability of \ref{ax:surj} is shown as follows.
\begin{enumerate}
 \item \label{itm:psa}
 \derline{\ref{ax:surj}$(g)$, \ref{ax:compat}$(f)$}
 {$\rho x\ent_x\exists u\qdot \sigma u\wedge gu=fx$}
 \item \label{itm:psb}
 \derline{}
 {$\rho(hp^*),\sigma(kp^*)\ent_{p^*}\rho(hp^*)\wedge\sigma(kp^*)\wedge hp^*=hp^*$}
 \item \label{itm:psc}
 \derline{\ref{itm:psb}}
 {$\rho(hp^*),\sigma(kp^*)\ent_{p^*}\exists p\qdot\rho(hp)\wedge\sigma(kp)\wedge hp=hp^*$}
 \item \label{itm:psd}
 \derline{\ref{itm:psc}}
 {$\rho x,\sigma u,gu=fx\ent_{x,u}\exists p\qdot\rho(hp)\wedge\sigma(kp)\wedge hp=x$}
 \item \label{itm:pse}
 \derline{\ref{itm:psa}, \ref{itm:psc}}
 {$\rho x\ent_x\exists p\qdot \rho(hp)\wedge\sigma(kp)\wedge hp=x$}
\end{enumerate}
Pullback stability of \ref{ax:inj} is shown as follows.
\begin{enumerate}
 \item \label{itm:gba}
 \derline{\ref{ax:inj}$(g)$}
{$\sigma(kp),\sigma(kq),\tau(gkp,gkq)\ent_{p,q} \sigma(kp,kq)$}
 \item \label{itm:gbb}
 \derline{\ref{ax:compat}$(f)$, $fh=gk$}
{$\rho(hp,hq)\ent_{p,q} \tau(gkp,gkq)$}
 \item \label{itm:gbc}
 \derline{\ref{itm:gba}, \ref{itm:gbb}}
{$\sigma(kp),\rho(hp,hq),\sigma(kq)\ent_{p,q} \sigma(kp,kq)$}
 \item \label{itm:gbd}
 \derline{\ref{itm:gbc}}
{$\rho(hp),\sigma(kp),\rho(hp,hq),\rho(hq),\sigma(hq)\ent_{p,q}\rho(hp,hq)\wedge\sigma(kp,kq) $}
\end{enumerate}
A path object for $\aro$ is given by 
\begin{equation}\label{eq:pathobj-fct}
\aro\xrightarrow{s}(A\times A,\tilde\rho)\xrightarrow{d}\aro\times\aro=
(A\times A,\rho\brprod\rho)
\end{equation}
with 
$
\tilde\rho(u,v)\equiv(\rho\brprod\rho)|_\rho
$ and where the underlying maps of $s$ and $d$ are $\delta$ and $\id$,
respectively. Then $d$ is a fibration by Lemma~\ref{lem:restrict-fib}, and it is
easy to see that $s$ is compatible with $\rho$ and $\tilde\rho$ and a weak
equivalence.

Finally, it is straightforward to check that terminal projections $\aro\to 1$ are 
fibrations.
\end{proof}

\begin{remark}\label{rem:degenerate}
The fibration part of all path object
factorizations~\eqref{eq:pathobj-fct} is monic, since the underlying map is iso
and the forgetful functor reflects monomorphisms. This implies that the
\emph{$\infty$-localization of $\canp$}\,---\,\ie the finitely complete
$\infty$-category obtained by weakly inverting weak equivalences, see
\eg~\cite{kapulkin2017quasicategories}\,---\,has the property that all of its
hom-spaces are discrete, or equivalently that all of its objects are
$0$-truncated. Indeed, if the second factor of a path object factorization $X\to
PX\to X\times X$ is monic, then
\begin{equation}
  \begin{tikzcd}[sep=small]PX\ar[r]\ar[d]&PX\ar[d]\\PX\ar[r]&X\times
  X\end{tikzcd}
\end{equation}
is a pullback of fibrations and therefore a homotopy pullback, which means that
$PX\to X\times X$ is a homotopy embedding, and therefore so is the diagonal
$X\to X\times X$.
\end{remark}

\section{The homotopy category}

Throughout this section let $\trip:\catcpos$ be a fixed  indexed frame. In the
following we show that the homotopy category of $\canp$ is the category
$\catc[\trip]$ of partial equivalence relations and \emph{functional relations}
in $\trip$. 
We start by recalling the definition from~\cite[Def.~3.1]{pitts2002}.
\begin{definition}
The category $\catc[\trip]$ has the same objects as $\canp$, and morphisms from
$\aro$ to $\bsi$ are predicates $\phi\in\trip(A\times B)$ satisfying the
judgments
\begin{description}
\item[\namedlabel{ax:strict}{(strict)}]
$\phi(x,u)\ent_{x,u}\rho x\wedge \sigma u$
\item[\namedlabel{ax:cong}{(cong)}]
$\rho(y,x),\phi(x,u),\sigma(u,v)\ent_{x,y,u,v}\phi(y,v)$
\item[\namedlabel{ax:singval}{(singval)}]
$\phi(x,u),\phi(x,v)\ent_{x,u,v}\sigma(u,v)$
\item[\namedlabel{ax:tot}{(tot)}]
$\rho x\ent_x\exists u\qdot \phi(x,u)$.
\end{description}
Composition of morphisms 
$\aro\xrightarrow{\phi}\bsi\xrightarrow{\gamma}\cta$ is given by 
\begin{equation}
(\gamma\circ\phi)(x,r)\;\equiv\;\exists u\qdot \phi(x,u)\wedge\gamma(u,r).
\end{equation}
The identity morphism on $\aro$ is given by the predicate $\rho$ itself.
\end{definition}
\begin{remarks}
\begin{anumerate}
\item As already mentioned, $\csqp$ is an exact category, and Maietti and
Rosolini have shown that it is in fact the `free' one on
$\trip$~\cite[Corollary~3.4]{maietti2012unifying}. 

\item Besides the presentation as homotopy category of $\canp$ (which depends on
the weak equivalences as additional data), the category $\csqp$ can also be
presented as category of functional relations in the sense
of~\cite[Section~3]{maietti2017triposes} in the indexed poset
$\widetilde{\trip}:\canp\op\to\catpos$ of descent predicates from
Remark~\ref{rem:desc}\ref{rem:desc:ipos}. This works because $\widetilde{\trip}$
is an indexed frame whenever $\trip$ is.

\item If $\trip$ is the canonical indexing of a frame $H$, then
$\catset[\trip]=\catset[\fam(H)]$ is the category $\sh(H)$ of sheaves on
$H$~\cite[Section~2.8]{pitts81}. 
\end{anumerate}
\end{remarks}
The following lemma
recalls constructions of finite limits and characterizations of monomorphisms
and covers in $\csqp$.
\begin{lemma}\label{lem:frel-iso}
\begin{inumerate}
\item Finite products in $\csqp$ are given by $1=(1,\top)$ and
$\aro\times\bsi=(A\!\times\! B,\rho\!\brprod\!\sigma)$ as in $\canp$. An
equalizer of $\phi,\gamma:\aro\to\bsi$ is given by 
\begin{equation}
(A,\rho|_\upsilon)\xrightarrow{\rho|_\upsilon}\aro\qtext{where}
\upsilon(a)\equiv \ex b\qdot \phi(a,b)\wedge\gamma(a,b).
\end{equation}
\item\label{lem:frel-iso-iso}
A map $\phi:\aro\to\bsi$ in $\catc[\trip]$ is a monomorphism \iff 
the judgment
\begin{description}
\item[\namedlabel{ax:inj'}{(inj*)}]
$\phi(x,u),\phi(y,u)\ent_{x,y,u}\rho(x,y)$
\end{description}
holds in $\trip$. It is a \emph{cover}\footnote{Covers are maps that do not
factor through any proper subobject of their codomain. In regular categories
they coincide with regular epimorphisms. See~\cite[A1.3]{elephant1}.} \iff 
\begin{description}
\item[\namedlabel{ax:esurj'}{(esurj*)}]
$\sigma u\ent_u\exists x\qdot \phi(x,u)$
\end{description}
holds in $\trip$. In particular, $\phi$ is an isomorphism \iff both \ref{ax:inj'} 
and \ref{ax:esurj'} hold.
\end{inumerate}
\end{lemma}
\begin{proof}
These statements can be found in \cite[Sections 2.4, 2.5]{pitts81}  and
\cite[Section~2.2]{vanoosten2008realizability} for triposes. The latter
reference explicitly uses the word `cover' whereas the former speaks about
general {epimorphisms}, which are known to coincide with regular
epimorphisms\,---\,\ie covers\,---\,in toposes. 
In \cite[Section~2.2.1]{frey2013fibrational} the same statements can be found
for `existential fibrations', which correspond to indexed frames under the
equivalence of indexed posets and posetal fibrations.
In all cases the statements are given without proofs, which are considered
straightforward.
\end{proof}

\begin{definition}\label{def:functor-e} The functor 
\begin{equation}
E:\canp\to\catc[\trip]
\end{equation} 
is the identity on objects, and is defined by  
\begin{equation}
 E(f)(x,u)\;\equiv\; \rho(x)\wedge\sigma(f x,u)
\end{equation}
for morphisms $f:\aro\to\bsi$.
\end{definition}

\begin{theorem}\label{thm:canp-hocat-catcp}
\begin{inumerate}
\item\label{thm:canp-hocat-catcp-charwe} A morphism $f:\aro\to\bsi$ in $\canp$ is a weak equivalence if and only if
$E(f)$ is an isomorphism in $\catc[\trip]$.
\item\label{thm:canp-hocat-catcp-loc} For any category $\catd$ and any
functor $F:\canp\to\catd$ sending weak equivalences to 
isomorphisms there exists a unique $\widetilde{F}:\catc[\trip]\to\catd$
satisfying $\widetilde{F}\circ E=F$.
\end{inumerate}
\end{theorem}
\begin{proof}
The first claim follows from Lemma~\ref{lem:frel-iso} and the facts that
\ref{ax:inj} holds for $f$ if and only if \ref{ax:inj'} holds for $ E(f)$, and 
that \ref{ax:esurj} holds for $f$ if and only if \ref{ax:esurj'} holds for 
$ E(f)$, as is easily verified.

For the second claim assume that $F:\canp\to\catd$ inverts weak equivalences.
Since $E$ is identity-on-objects, we only have to define $\widetilde{F}$ on
morphisms. Let $\phi:\aro\to\bsi$ in $\catcp$. We construct the span
\begin{equation}\label{eq:faq1}
\aro\xleftarrow{\phi_l}
(A\times B,(\rho\brprod\sigma)|_\phi)\xrightarrow{{\phi_r}}\bsi
\end{equation}
in $\canp$, where the underlying functions of $\phi_l$ and $\phi_r$ are the
projections, and $(\rho\brprod\sigma)|_\phi$ is as in
Definition~\ref{def:restrictper}. Then one easily verifies that $\phi_l$ is a
trivial fibration, and that 
\begin{equation}
  \phi\circ E(\phi_l) = E(\phi_r) 
\end{equation}
in $\catcp$.
For any $\widetilde{F}$ satisfying $\widetilde{F}\circ {E}=F$ we therefore must have 
\begin{equation}
  \widetilde{F}(\phi)\circ F(\phi_l) = F(\phi_r)
\end{equation}
and since $F$ is assumed to invert weak equivalences we can deduce
\begin{equation}\label{eq:functor-tilde-f}
  \widetilde{F}(\phi) = F(\phi_r)\circ F(\phi_l)\inv.
\end{equation}
It remains to show that~\eqref{eq:functor-tilde-f} defines a functor satisfying
$\widetilde F\circ E = F$.

To see that the construction commutes with composition, let 
\begin{equation}\aro\xrightarrow{\phi}\bsi\xrightarrow{\gamma}\cta\end{equation}
in $\catcp$, and define $\xi\in\trip(A\times B\times C)$ 
and $\theta\in\trip(A\times C)$ by 
\begin{align}
\xi(x,u,r) \;&\equiv\; \phi(x,u)\wedge\gamma(u,r)\quad\text{and}\\
\theta(x,r)\;&\equiv\;\exists u\qdot \xi(x,u,r),
\end{align}
in other words $\theta=\gamma\circ\phi$. Consider the following diagram.
\begin{equation}
\begin{tikzcd}[sep = small]
&   (A\!\times\! C,(\rho\!\brprod\!\tau)|_\theta)
        \ar[dl,"{\theta_l}"',"{\sim}"]
        \ar[dr,"\theta_r"]
\\  \aro 
&&  \cta 
\\& (A\smtimes B\smtimes C,(\rho\!\brprod\!\sigma\!\brprod\!\tau)|_\xi)
        \ar[uu,"\partial_1","\sim"']
        \ar[ld,"\partial_2","\sim"']
        \ar[rd,"\partial_0"']
\\  (A\smtimes B,(\rho\!\brprod\!\sigma)|_\phi)
        \ar[uu,"\phi_l","\sim"']
        \ar[rd,"\phi_r"']
&&  (B\!\times\! C,(\sigma\!\brprod\!\tau)|_\gamma)
        \ar[uu,"\gamma_r"']
        \ar[ld,"\gamma_l","\sim"']
\\& \bsi
\end{tikzcd}
\end{equation}
The three squares commute since the underlying maps are simply projections,
$\phi_l$, $\gamma_l$, and $\theta_l$ are trivial fibrations as remarked earlier,
and moreover the upper left square is a pullback, whence $\partial_1$ and
$\partial_2$ are trivial fibrations, in particular weak equivalences, as
well\footnote{Alternatively one can verify by hand that $\partial_1$ and
$\partial_2$ are weak equivalences and skip the pullback argument, to obtain a
proof that does not depend on the fibrations in the
category-of-fibrant-objects-structure on $\canp$ (\cf discussion in the
introduction).}. Applying $F$ we can argue
\begin{align}
\widetilde F(\gamma)\circ\widetilde F(\phi)
=\;&F(\gamma_r)\circ F(\gamma_l)\inv\circ F(\phi_r)\circ F(\phi_l)\inv\\
=\;&F(\gamma_r)\circ F(\partial_0)\circ F(\partial_2)\inv\circ F(\phi_l)\inv\\
=\;&F(\theta_r)\circ F(\partial_1)\circ F(\partial_2)\inv\circ F(\phi_l)\inv\\
=\;&F(\theta_r)\circ F(\theta_l)\inv = \widetilde F(\theta) = \widetilde F(\gamma\circ\phi).
\end{align}
Preservation of identities is straightforward and thus $\tilde F$ is
functorial.

To see that $\widetilde F\circ E = F$, let $f:\aro\to\bsi$ in $\canp$ and
consider the diagram
\begin{equation}\label{eq:cofo-fac}
\begin{tikzcd}
    (A\smtimes B,(\rho\brprod\sigma)|_{ E(f)}) 
        \ar[d, bend right, "E(f)_l"']
        \ar[dr,"E(f)_r"]
\\  \aro
        \ar[r,"f" near start]
        \ar[u,bend right,"s"]
&   \bsi
\end{tikzcd}
\end{equation}
in $\canp$, where $s$ has underlying map $\langle\id_A,f\rangle$. Then $ E(f)_r\circ s =f$
and $s$ is a section of the weak equivalence $ E(f)_l$, which means that $F(s)$
is an inverse of $F( E(f)_l)$, and we can argue
\begin{equation}
\widetilde F(E(f)) = F( E(f)_r)\circ F( E(f)_l)\inv= F( E(f)_r)\circ F(s)
= F(f)
\end{equation}
as required.
\end{proof}
\begin{remark}\label{rem:loc-thm}
\begin{anumerate}
\item\label{rem:loc-thm-hocat} 
Theorem~\ref{thm:canp-hocat-catcp}\ref{thm:canp-hocat-catcp-loc} says that
the functor $E:\canp\to\catcp$ exhibits $\catcp$ as the \emph{localization}
$\canp[\cweq\inv]$ of $\canp$ by the class $\cweq$ of weak equivalences.
See~\cite[Section~I.1]{gabriel1967calculus},
\cite[I.2.2(ii)]{dwyer2005homotopy}.
Following~\cite[I.2.3(iv)]{dwyer2005homotopy} we also use the term
\emph{homotopy category} instead of `localization'.
\item
The factorization of $f$ displayed in~\eqref{eq:cofo-fac} is an instance of the
standard \emph{mapping cocylinder factorization}, which factors morphisms
$h:X\to Y$ in a category of fibrant objects into a section $s$ of a trivial
fibration followed by a fibration $d_1\circ h'$, as in the following diagram
where the square is a pullback~\cite[Factorization Lemma
pg.~421]{brown1973abstract}.
\begin{equation}\begin{tikzcd}
    \cocyl(h)
        \ar[r,"h'"']
        \ar[d,"", ]
&	P Y
        \ar[d,"d_0"]
        \ar[r,"d_1"']
&   Y
\\	X
        \ar[r,"h"]
        \ar[u,bend left,"s"]
&	Y
\end{tikzcd}\end{equation}
In particular we have 
$(A\smtimes B,(\rho\brprod\sigma)|_{ E(f)})=\aro\times_{\bsi}P\bsi$.
\end{anumerate}\end{remark}
The following lemma characterizes the \emph{kernel} of the localization functor
$E$.
\begin{lemma}\label{lem:homotopy}
For parallel arrows $f,g:\aro\to\bsi$ in $\canp$, the
following are equivalent:
\begin{inumerate}
\item \label{lem:homotopy-efeg}
$E(f)=E(g)$
\item \label{lem:homotopy-judg}
The judgment $\rho x\ent_x \sigma(fx,gx)$ holds.
\item \label{lem:homotopy-pobj}
 $\langle f,g\rangle$ factors through the path
object from~\eqref{eq:pathobj-fct}.   
\begin{equation}
  \begin{tikzcd}[column sep = large]
    & (B\smtimes B,\tilde{\sigma})
    \ar[d,"d"]\\
  \aro
  \ar[r,"{\langle f,g\rangle}"]
  \ar[ru,dashed, shift left = 2] 
  & \bsi\times\bsi
  \end{tikzcd}
\end{equation}
\end{inumerate}
\end{lemma}
\begin{proof}
  Easy.
\end{proof}
\begin{remarks}
\begin{anumerate}
\item
  In general \cofos, the equivalence between conditions \ref{lem:homotopy-efeg}
  and \ref{lem:homotopy-pobj} of Lemma~\ref{lem:homotopy} is replaced by the
  more complicated statement that parallel arrows $f,g:A\to B$ are identified in
  the homotopy category whenever there exists an equivalence $e:A'\to A$ such that
  $\langle f,g\rangle\circ e$ factors through a path object
  (see~\cite[Theorem~1-(ii)]{brown1973abstract}).
  \item  The construction of the homotopy category of a \cofo $\mcc$ given
  in~\cite{brown1973abstract} proceeds in two steps: first one defines a
  category $\pi(\mcc)$ by quotienting the morphisms of $\mcc$ by the relation
  described in item 1, and then $\ho(\mcc)$ is obtained by localizing
  $\pi(\mcc)$ by a calculus of fractions. This two-step construction gives rise
  to a factorization 
  \begin{equation}
    \mcc\to\pi(\mcc)\to\ho(\mcc)
  \end{equation}
  of the localization functor into a full functor followed by a faithful
  functor (both identity-on-objects).
  
  When applying this factorization to the functor $\canp\to\catcp$ in the case
  where $\trip$ is a tripos\,---\,\ie we quotient $\canp$ by the congruence
  relation analyzed in Lemma~\ref{lem:homotopy}\,---\,we recover in the middle
  the \emph{q-topos} $\mathbf{Q}(\trip)$ described
  in~\cite[Definition~5.1]{frey2015triposes}.
  \begin{equation}
    \canp\to\mathbf{Q}(\trip)\to\catcp
  \end{equation}
  In particular if $\trip$ is the canonical indexing of a frame $A$ then the
  middle category is equivalent to the quasitopos of separated presheaves.
  \begin{empheq}{equation}
    \catset\langle\fam(A)\rangle\to\mathbf{Sep}(A)\to\sh(A)
  \end{empheq}
\end{anumerate}
\end{remarks}

\section{Cofibrant objects and \texorpdfstring{$\ex$}{∃}-completions}
\label{sec:eprime}

\begin{definition} Let $\catc$ be a finite-limit category.
\begin{anumerate}
\item Let $\fifd:\catcpos$ be an indexed poset, let $I\in\catc$, and let 
$\varphi_1,\dots,\varphi_n\in\fifd(I)$ ($n\geq 0$) be a list of predicates.

We define $\subslice{\fifd}{I}{(\varphi_1,\dots,\varphi_n)}$ to be the category
whose objects are pairs $(f:J\to I,\psi\in\fifd(J))$ satisfying $\psi\leq
f^*\varphi_i$ for $1\leq i\leq n$, and whose morphisms from $(f: J\to I,\psi)$
to $(g:K\to I,\theta)$ are arrows $h:J\to K$ such that $g\circ h=f$ and
$\psi\leq h^*\theta$.
\item   
An indexed poset $\fifd:\catcpos$ is said to satisfy the \emph{solution set
condition for finite meets (\ssc)}, if for all $I\in\catc$ and
$\varphi_1,\dots,\varphi_n\in\fifd(I)$ the category
$\subslice{\fifd}{I}{(\varphi_1,\dots,\varphi_n)}$ has a weakly terminal object.
\item An indexed monotone map $\Phi:\fifd\to\trip:\catcpos$ from an indexed
poset $\fifd$ satisfying the \ssc to an indexed frame $\trip$ is called
\emph{flat}, if 
\begin{equation}\ex_f\Phi_J\psi\geq\Phi_I\varphi_1\wedge\dots\wedge\Phi_I\varphi_n\end{equation} whenever
$(f:J\to I,\psi)$ is weakly terminal in $\subslice{\fifd}{I}{(\varphi_1,\dots,\varphi_n)}$.
(The converse inequality always holds.)

\item A flat map $\Phi:\fifd\to\trip$ is called an \emph{\ecompletion of
$\fifd$} if precomposition induces an order isomorphism 
\begin{equation}
\mathbf{IFrm}(\trip,\triq)\xto{\cong}\mathbf{Flat}(\fifd,\triq)
\end{equation}
between morphisms of indexed frames $\trip\to\triq$ and flat maps
$\fifd\to\triq$.
\item A predicate $\varpi\in\trip(I)$ of an indexed frame $\trip$ is
  called \emph{\eprime}, if for every composable pair $I
  \xleftarrow{u}J\xleftarrow{v}K$ of maps and every $\psi\in\trip(K)$ satisfying
  $u^*\varpi\leq\exists_v\psi$, there exists a section $s$ of $v$ such that
  $u^*\varpi\leq s^*\psi$.
\item 
  We say that an indexed frame $\trip$ has \emph{enough} \eprime predicates, if
  for every predicate $\varphi\in\trip(I)$ there is an \eprime predicate
  $\varpi\in\trip(J)$ and a map $f : J\to I$ such that $\varphi = \ex_f\varpi$.
\end{anumerate}
\end{definition}
It is clear that \ecompletions of a given $\fifd$ are unique up to isomorphism
whenever they  exist, and that \eprime predicates are stable under reindexing in
any  indexed frame. The term `flat' is justified by the following lemma.
\begin{lemma} 
\begin{inumerate}
\item  Every \imsl satisfies the \ssc. 
\item  An indexed monotone map $\Phi:\fifd\to\trip:\catcpos$ from an indexed
meet-semilattice to an indexed frame is flat if and only if  it is cartesian.
\end{inumerate}\end{lemma}
\begin{proof}
If $\fifd$ is an indexed semilattice then $(\id_I,\varphi_1\wots\varphi_n)$ is
terminal in $\subslice{\fifd}{I}{(\varphi_1,\dots,\varphi_n)}$, and with this it
is immediate that flat maps are cartesian. 

Conversely, assume that $\Phi$ is cartesian and that $(f:J\to I,\psi)$ is weakly
terminal in $\subslice{\fifd}{I}{(\varphi_1,\dots,\varphi_n)}$. Then the
terminal projection 
\begin{equation}
  (f,\psi)\to(\id_I,\varphi_1\wots\varphi_n)
\end{equation}
has a section $s$, and we can argue
\begin{empheq}{align}
\Phi_I\,\varphi_1\wots\Phi_I\,\varphi_n
\,&\leq\, s^*\Phi_J\,\psi
&&\text{since }\varphi_1\wots\varphi_n\leq
s^*\,\psi\\
\,&=\,\ex_f\sex_s\,s^*\Phi_J\,\psi
&&\text{since } f\circ s = \id_I\\
\,&\leq\,\ex_f\,\Phi_J\,\psi
 && \text{since }\ex_s \adj s^*.
\end{empheq}
\end{proof}
\begin{theorem}\label{thm:flat-ecompletion}
\begin{inumerate}
\item \label{thm:flat-ecompletion-then-ecompletion}
Let $\trip$ be an indexed frame, and $\fifd\subs\trip$ an indexed subposet such 
that
\begin{itemize}
  \item the predicates in $\fifd$ are \eprime in $\trip$, and 
  \item for every $\varphi\in\trip(I)$ there are $f:J\to I$ and
$\varpi\in\fifd(J)$ with $\ex_f\,\varpi=\varphi$.
\end{itemize}
  Then $\fifd$ satisfies the \ssc and the inclusion
$\fifd\to\trip$ is an \ecompletion.

In particular, the inclusion of the indexed subposet of \eprime predicates
is an \ecompletion whenever $\trip$ has enough \eprime predicates.

\item \label{thm:flat-ecompletion-ex}
Every indexed poset $\fifd$ satisfying the \ssc has an \ecompletion.
\item \label{thm:flat-ecompletion-if-ecompletion} If
$\Phi:\fifd\to\trip:\catcpos$ is an \ecompletion, then $\Phi$ is fiberwise
order-reflecting and its image consists precisely of the \eprime predicates in
$\trip$.
\end{inumerate}
\end{theorem}
\begin{proof}

\emph{Ad \ref{thm:flat-ecompletion-then-ecompletion}.} By assumption, given $\varpi_1\cots\varpi_n\in\fifd(I)$ there
 exist $g:K\to I$ and $\upsilon\in\fifd(K)$ such that
 $\ex_g\upsilon=\varpi_1\wots\varpi_n$, and we claim that
 $(g, \upsilon)$ is weakly terminal in $\subslice{\fifd}{I}{(\varpi_1,\dots,\varpi_n)}$. To see this let $f:J\to I$ and $\pi\in\fifd(J)$ with $\pi\leq
 f^*(\varpi_i)$ for $1\leq i\leq n$.
 Let $J\xleftarrow{\tilde g}L\xrightarrow{\tilde f}K$ be a pullback of $J\xrightarrow{j}I\xleftarrow{g}K$. Then we can argue
 \begin{equation}
 \pi\leq f^*(\varpi_1\wedge\dots\wedge\varpi_n)=
 f^*\,\ex_g\,\upsilon 
= \ex_{\tilde g}\,\tilde f^*\,\upsilon
 \end{equation}
and since $\pi$ is \eprime, $\tilde g$ has a section $s$ such that $\pi\leq s^*\,\tilde f^*\,\upsilon$. It follows that $\tilde f\circ s$ constitutes the required morphism from  $(f,\pi)$ to $(g,\upsilon)$.

To see that $\fifd\to\trip$ is
an \ecompletion consider a flat map $\Phi:\fifd\to\triq$ and define
$\Psi:\trip\to\triq$ by 
\begin{equation}\label{eq:Psi-def}
  \Psi_I\,\varphi=\ex_f\,\Phi_J\,\pi
\end{equation}
 where
$\pi\in\fifd(J)$ with $\ex_f\pi=\varphi$. To show that $\Psi$ is fiberwise
monotone\,---\,and in particular well-defined\,---\,let $\varphi\leq\psi\in\trip(I)$ and let $\pi\in\fifd(J)$ and
$\upsilon\in\fifd(K)$ such that 
$\ex_f\pi=\varphi$ and $\ex_g\upsilon=\psi$.
Again let $J\xleftarrow{\tilde g}L\xrightarrow{\tilde f}K$ be a pullback of the span
$J\xrightarrow{f}I\xleftarrow{g}K$. 
We have 
\begin{equation}
\ex_f\,\pi\leq\ex_g\,\upsilon\qtext{which implies}\pi\leq f^*\,\ex_g\,\upsilon
=\ex_{\tilde g}\,\tilde f^*\,\upsilon,
\end{equation}
and since $\pi$ is \eprime, $\tilde g$ has a section $s$ with $\pi\leq
s^*\,\tilde f^*\upsilon$, whence we can deduce 
\begin{equation}
  \ex_f\,\Phi_J\,\pi
=  \ex_{g\,\tilde f\,s}\,\Phi_J\,\pi
\leq  \ex_{g}\,\ex_{\tilde f\,s}\,\Phi_J\,(\tilde f\,s)^*\upsilon
\leq \ex_{g}\,\Phi_K\,\upsilon.
\end{equation}
Naturality of $\Psi$ follows from~\ref{ax:bc} and preservation of $\ex$ is
straightforward. 

To see that $\Psi$ preserves binary meets, let $\varphi,\psi\in\trip(I)$ and
consider maps $f:J\to I$, $g:K\to I$ and predicates $\pi\in\fifd(J)$,
$\upsilon\in\fifd(K)$ such that 
\begin{equation}\label{eq:varpi-upsilon-sth}
  \ex_f\pi=\varphi\qtext{and} \ex_g\upsilon=\psi.
\end{equation}
Let $J\xleftarrow{\tilde g}L\xrightarrow{\tilde f}K$ be a pullback 
of the span $J\xrightarrow{f}I\xleftarrow{g}K$, and let $h:M\to L$ and
$\mu\in\fifd(M)$ such that 
\begin{equation}\label{eq:mu-sth}
  \ex_h\mu=\tilde g^*\pi\wedge\tilde f^*\upsilon.
\end{equation} 
Then $(h,\mu)$
is weakly terminal in $\subslice{\fifd}{L}{(\tilde g^*\pi,\tilde
f^*\upsilon)}$ and 
since $\Phi$ is flat we have 
\begin{equation}\label{eq:since-phi-flat}
\ex_{h}\,\Phi_M\,\mu=\tilde g^*\Phi_J\,\pi\wedge\tilde f^*\,\Phi_K\,\upsilon.
\end{equation}
Moreover, we have 
\begin{equation}\label{eq:mu-covering}
  \begin{aligned}
    \ex_f\sex_{\tilde g}\sex_h\,\mu
    & =  \ex_f\sex_{\tilde g}(\tilde g^*\,\pi\wedge\tilde f^*\,\upsilon)
    &&\text{by \eqref{eq:mu-sth}}\\
    & =  \ex_f(\pi\wedge\ex_{\tilde g}\,\tilde f^*\upsilon)
    &&\text{by \ref{ax:fr}}\\
    & =  \ex_f(\pi\wedge f^*\sex_{g}\,\upsilon)
    &&\text{by \ref{ax:bc}}\\
    & =  \ex_f\,\pi\wedge \ex_{g}\,\upsilon
    &&\text{by \ref{ax:fr}}\\
    & = \varphi\wedge\psi&&\text{by \eqref{eq:varpi-upsilon-sth}},
  \end{aligned}
  \end{equation}
and thus we can argue
\begin{align}
\Psi_I(\varphi\wedge\psi)
&=\Psi_I\,\ex_f\sex_{\tilde g}\sex_h\,\mu
&&\text{by \eqref{eq:mu-covering}}\\
&=\ex_f\sex_{\tilde g}\sex_h\,\Psi_M\,\mu
&&\text{since $\Psi$ commutes with $\ex$}\\
&=\ex_f\sex_{\tilde g}(\tilde g^*\Phi_J\,\pi\wedge\tilde f^*\,\Phi_K\,\upsilon)
&&\text{by \eqref{eq:since-phi-flat} and since $\Phi=\Psi$ on $\fifd$}\\
&=\ex_f(\Phi_J\,\pi\wedge\ex_{\tilde g}\,\tilde f^*\,\Phi_K\,\upsilon)
&&\text{by \ref{ax:fr}}\\
&=\ex_f(\Phi_J\,\pi\wedge f^*\,\ex_{ g}\,\Phi_K\,\upsilon)
&&\text{by \ref{ax:bc}}\\
&=\ex_f\,\Phi_J\,\pi\wedge \ex_{ g}\,\Phi_K\,\upsilon&&\text{by \ref{ax:fr}}\\
&=\Psi_I\,\varphi\wedge\Psi_I\,\psi&&\text{by \eqref{eq:Psi-def}}.
\end{align}
The fact that $\Psi$ preserves $\top$ is shown along the same lines.

\medskip

\emph{Ad \ref{thm:flat-ecompletion-ex}.} Given $\fifd:\catcpos$ satisfying the \ssc, define
$\trip:\catcpos$ to be the indexed poset whose predicates on $I$ are 
equivalence classes of pairs $(f:J\to I, \varphi\in\fifd(J))$, ordered by 
setting
$  (J\xto{f} I,\varphi)\leq(K\xto{g} I,\psi)$ iff there exists an 
  $h:J\to K$ with  $g h=f$ and $\varphi\leq h^*\,\psi$,
and quotiented by the symmetric part of this relation. Reindexing is defined by
$g^*[(f,\varphi)]=[(\tilde f,\tilde g^*\varphi)]$ where $\cdot\xleftarrow{\tilde
f}\cdot\xrightarrow{\tilde g}\cdot$ is a pullback of
$\cdot\xrightarrow{g}\cdot\xleftarrow{f}\cdot$, and existential quantification
is given by $\ex_g[(f,\varphi)]=[(g\circ f,\varphi)]$. A binary meet of
$[(f:J\to I,\varphi)]$ and $[(g:K\to I,\psi)]$ is given by $[(f\circ\tilde
g\circ h,\theta)]$ where $(h:L\to J\times_I K,\theta)$ is weakly terminal in
$\subslice{\fifd}{J\times_I K}{(\tilde g^*\varphi,\tilde f^*\psi)}$, and the
greatest element of $\trip(I)$ consists of the weakly terminal objects of
$\subslice{\fifd}{I}{()}$. The verifications of \ref{ax:fr} and \ref{ax:bc} are
straightforward. Thus $\trip$ is an indexed frame, and it is easy to see that the
assignment $\varphi\mapsto(\id,\varphi)$ defines an order-reflecting indexed
monotone map $\fifd\to\trip$ satisfying the conditions of
\ref{thm:flat-ecompletion-then-ecompletion}.

\medskip

\emph{Ad~\ref{thm:flat-ecompletion-if-ecompletion}.} From the construction in
the proof of \ref{thm:flat-ecompletion-ex} we know that up to isomorphism every
\ecompletion $\fifd\to\trip$ is of the form described in
\ref{thm:flat-ecompletion-then-ecompletion}, in particular it is
order-reflecting and its image consists of \eprime predicates. Moreover, for all
$\varphi\in\trip(I)$ there exist $f:J\to I$ and $\varpi\in\fifd(J)$ such that 
\begin{equation}\label{eq:swow}
\ex_f\varpi=\varphi,
\end{equation}
and if $\varphi$ is \eprime then $f$ has a section with
$\varphi\leq s^*\varpi$ (by $\ex$-primality) and
$s^*\varpi\leq\varphi$ (follows from~\eqref{eq:swow}) which shows that 
all \eprime predicates are contained in $\fifd$.
\end{proof}
\begin{examples}\label{ex:eprime} For the statements about $\catset$-indexed
posets in the following examples we assume the axiom of choice, so that all
surjections split. To obtain analogous statements without choice we have to work
with indexed posets that satisfy a \emph{stack condition} and adapt the
definitions of `\ecompletion' and `\eprime' accordingly (see
\cite[Section~3.4]{frey2013fibrational}).
\begin{anumerate}
\item\label{ex:eprime:real} Realizability
triposes~\cite{vanoosten2008realizability} have enough \eprime predicates.
Specifically, the \eprime predicates on a set $I$ in the tripos over a PCA
$\pcaa$ are precisely the (equivalence classes of) families $\varphi:I\to
P(\pcaa)$ of subsets of $\pcaa$ which consist only of singletons. Analogous
statements are true for \emph{relative realizability} and realizability over
\emph{ordered} PCAs~\cite{hofstra2006all}.
\item\label{ex:eprime:low} The canonical indexing $\fam(P):\catset\optopos$ of a
small poset $P$ always satisfies the \ssc, its \ecompletion is given by
\begin{equation}\fam(P)\to\fam(\low(P))\end{equation} 
where $\low(P)$ is  the frame of lower
sets in $P$.
\item 
The canonical indexing of a frame $A$ has enough \eprime predicates
precisely if $A$ is of the form $\low(P)$ for some poset $P$, which in turn
is equivalent to $A$ having enough \emph{completely join prime elements}~\cite[Proposition~1.1]{davey1979lattice}.
This observation is the reason for the choice of terminology.
\item\label{ex:eprime:finrestrict} 
The restriction of the canonical indexing of a poset $P$ to the 
category of finite sets satisfies the \ssc precisely if every finite 
intersection of principal ideals in $P$ is also a finite union of principal
ideals. In this case, the \ecompletion of the restriction is the 
restriction of the canonical indexing of the completion of $P$ under finite
suprema, \ie the subposet of $\low(P)$ on finite unions of principal ideals.
\item Analogously to \ref{ex:eprime:finrestrict}, the canonical 
indexing of a \emph{large} poset $P$ satisfies the \ssc \iff the 
small-join completion of $P$ is closed under intersections.
\end{anumerate}
\end{examples}
\begin{remarks}
\begin{anumerate}
\item
An indexed preorder $\fifd:\catcpos$ satisfies the \ssc if and only if its
total category $\gcons\fifd$ has \emph{weak finite limits}.

If this is the case and $\fifd\to\trip$ is an \ecompletion, then the functor
\begin{equation}
\gcons\fifd\to\catc[\trip],\qquad
(I,\varpi)\mapsto(I,{=}|_\varpi)
\end{equation} 
is an exact completion in the sense of~\cite[Theorem~29]{carboni1998regular}.
\item \ecompletions and \eprime predicates are treated
in~\cite[Section~3.4.2.1]{frey2013fibrational} for indexed posets that are
{pre-stacks} for the regular topology on a regular category. In this case, the
primality condition and the construction of the \ecompletion have to be stated
slightly differently to take the topology into account.

In~\cite{trotta2020existential} Trotta introduces a notion of \ecompletion
relative to a class of maps that is closed under composition and pullbacks.
\item $\exists$-primality can be viewed as a kind of `choice principle'. In
particular, an indexed frame $\trip$ satisfies the `rule of choice'
from~\cite[Definition~5.5]{maietti2017triposes} \iff $\top\in\trip(1)$ is
\eprime.
\end{anumerate}
\end{remarks}
Following Baues~\cite[Section~I.1]{baues1989algebraic} we call an object $C$ of
a category of fibrant objects $\mcc$ \emph{cofibrant}, if every trivial
fibration $f:B\to C$ admits a section. $\mcc$ is said to have \emph{enough}
cofibrant objects if every $A\in\mcc$ admits a \emph{cofibrant replacement},
\ie a cofibrant object $C$ and a trivial fibration $f:C\to A$.
\begin{proposition}\label{prop:enough-cofibs} Let $\tripcoppos$ be an indexed
frame. 
\begin{inumerate}
\item If $\ctau\in\canp$ such that the support $\tau_0$ of $\tau$ is \eprime,
then $\ctau$ is cofibrant.
\item
If $\trip$ has enough $\ex$-prime predicates, then $\canp$ has enough cofibrant
objects.
\end{inumerate}
\end{proposition}
\begin{proof}
To see that $\cta$ is cofibrant, let $f:\bsi\to\cta$ be a trivial fibration.
From Lemma~\ref{lem:tfib-eq-inj-surj} and Remark~\ref{rem:surj-reform} we
know that $ \tau_0 \leq\ex_f\sigma_0$, and since $\tau_0$ is
$\ex$-prime $f$ has a section $s:C\to B$ such that $\varpi\leq
s^*\sigma_0$, \ie the judgment $\tau(c)\ent_c \sigma(sc)$ holds. The judgment
\ref{ax:compat} for $s$ then follows from this and \ref{ax:inj} for $f$. Thus,
$s$ constitutes a morphism of type $\cta\to\bsi$ in $\canp$, which gives the
required section.

For the second claim, if $\aro\in\canp$ and $\trip$ has enough \eprime
predicates then there exists an object $C\in\catc$, an $\ex$-prime predicate
$\varpi\in\trip(C)$, and a morphism $e:C\to A$ such that $\ex_e\varpi = \rho_0$.
Setting 
\begin{equation}
\tau(c,c')\sequiv{\varpi(c)\wedge\rho(ec,ec')}
\end{equation}
it is easy to see that $\tau$ is a partial equivalence relation with support
$\tau_0=\varpi$, and again using Remark~\ref{rem:surj-reform} that $e$
constitutes a trivial fibration from $\cta$ to $\aro$.
\end{proof}

\section{Kan extensions along localization functors}

Recall from~\cite[Section~I-2.3]{dwyer2005homotopy} that a \emph{category with
weak equivalences} (or \emph{\wecat}) is a category $\mcc$ equipped with a class
$\cweq$ of arrows, called \emph{weak equivalences}, which satisfies
axiom~\ref{ax:a} of the definition of categories of fibrant objects.

\smallskip

The \emph{homotopy category} of a we-category is the category
$\ho(\mcc)=\mcc[\mcw\inv]$ obtained by freely inverting all weakly equivalence,
meaning that it comes with a \emph{localization functor} $E:\mcc\to\ho(\mcc)$
which inverts all weak equivalences and is initial among functors doing
so, see the references in Remark~\ref{rem:loc-thm}\ref{rem:loc-thm-hocat}.
\begin{definition}\label{def:proto-fibrant}
Let $\mcc$ be a \wecat with localization functor
$E:\mcc\to\ho(\mcc)$.
\begin{anumerate}
\item We call $X\in\mcc$ \emph{proto-fibrant}, if
$\mcc(A,X)\xto{E_{A,X}}\ho(\mcc)(EA,EX)$
is surjective for all $A\in\mcc$.
\item We say that $\mcc$ has \emph{enough} proto-fibrant objects, if every
$A\in\mcc$ admits a weak equivalence $\iota:A\to\overline{A}$ into a
proto-fibrant object (called its \emph{proto-fibrant replacement}).
\end{anumerate}
\emph{Proto-cofibrant} objects, and \wecats with enough proto-cofibrant objects
are defined dually.
\end{definition}
\begin{examples}\label{ex:proto-fib-cofib}
If in a model category $\mcm$ all objects are cofibrant, then all fibrant
objects are proto-fibrant. This is because morphisms from cofibrant to fibrant
objects in $\ho(\mcm)$ can be represented as equivalence classes of morphisms in
$\mcm$ modulo homotopy.

Similarly, cofibrant objects in \cofos $\mcc$ are proto-cofibrant, since
morphisms $f:A\to B$ in $\ho(\mcc)$ can be represented as right fractions
$f=E(g)\circ E(t)\inv$ with $t$ a trivial fibration \cite[2nd remark after
Theorem~1]{brown1973abstract}, and $t$ splits whenever $A$ is cofibrant.
\end{examples}
\begin{lemma}\label{lem:qfib-kan}
Assume that $\mcc$ is a \wecat with enough \pfib objects, $E:\mcc\to\ho(\mcc)$
is its localization functor, and $F:\mcc\to\catd$ is a functor into an arbitrary
category such that for all parallel pairs of arrows $f,g:A\to X$ into a \pfib 
object $X$ we have 
\begin{equation}\label{eq:preserve-homotopies}
  E f=Eg\qtext{implies} Ff=F g.
\end{equation}
Then $F$ admits a left Kan extension along
$E$.
\end{lemma}
\begin{proof}
Using the dual of~\cite[Theorem~X.3.1]{maclanecwm} and the fact that $E$ is
identity-on-objects, it is sufficient to show that for each $A\in\mcc$ the
functor
\begin{equation}
(E/EA)\xrightarrow{U}\mcc\xrightarrow{F}\catd
\end{equation} 
has a colimit. Let $\iota:A\to\overline A$ be a \pfib replacement. Then for each
$f:EB\to EA$ there exists $f^\dagger:B\to\overline A$ such that $E f^\dagger =
E\iota\circ f$. We define a cocone $\eta:F\circ U\stackrel{.}{\to}F\overline A$
by letting $\eta_f = F(f^\dagger)$, which does not depend on the choice of
$f^\dagger$ and is natural by~\eqref{eq:preserve-homotopies}. We define $\kappa
=(E\iota)\inv$ and note that $\eta_\kappa=\id_{F\overline A}$ since
$\id_{\overline{A}}$ is a choice for $\kappa^\dagger$. Given another cocone
$\theta:F\circ U\stackrel{.}{\to}D$ the map $\theta_\kappa$ constitutes a cocone
morphism $\eta\to\theta$, and it can be seen to be the only one by inspecting
the naturality condition for this cocone morphism at $\kappa$.
\end{proof}
\begin{blank}{Comparing left and right Kan extensions}\label{bl:compare}
From the $\catcat$-enriched universal property of 
localizations~\cite[Lemma~1.2]{gabriel1967calculus} it follows that 
precomposition 
with localization functors $E:\mcc\to\ho(\mcc)$ is fully faithful for arbitrary 
target categories $\catd$. From this it is immediate that every functor 
$G:\ho(\mcc)\to\catd$ is both a left and a right Kan extension of $G\circ E$ 
along $E$.
\begin{equation}
G=E_*(G\circ E)=E_!(G\circ E)
\end{equation}
In particular, if $F:\mcc\to\catd$ inverts all weak equivalences in $\mcc$, 
then the unique $\tilde F:\ho(\mcc)\to\catd$ with $\tilde F\circ E=F$ is both
a left and a right Kan extension.

Now if $F:\mcc\to\catd$ is such that both Kan extensions $E_!F$ and $E_*F$ exist, 
the 
counit map $\ve:(E_*F)\circ E\to F$ gives rise by functoriality to a transformation
\begin{equation}\label{eq:xi}
\xi:E_*F=E_!(E_*F\circ E)\xrightarrow{E_*\ve} E_!F
\end{equation} 
from the \emph{right} to the \emph{left} Kan extension. This is called the
\emph{`points-to-pieces' transform} by Lawvere in the setting of axiomatic
cohesion~\cite{lawvere2007axiomatic}. If $F$ inverts weak equivalences, then
$\ve$ and therefore $\xi$ becomes invertible by the remarks above.

Given a we-category $\mcc$ with enough proto-fibrant and enough proto-cofibrant
objects, and a functor $F:\mcc\to\catd$ such that $E(f)=E(g)$ {implies}
$F(f)=F(g)$ for all parallel $f,g$ with either cofibrant domain or fibrant
codomain\,---\,such that both Kan extensions exist by
Lemma~\ref{lem:qfib-kan}\,---\,there is an easy direct description of the
transformation~\eqref{eq:xi}: Given $A\in\mcc$ with proto-cofibrant replacement
$\kappa:\underline{A}\to A$ and proto-fibrant replacement
$\iota:A\to\overline{A}$, we have $E_*FA=F\underline{A}$, $E_!FA=F\overline{A}$
and 
\begin{equation}
\xi_A\;=\;E(\underline{A}\xrightarrow{\kappa}A\xrightarrow{\iota}\overline{A})
\;:\; E_*FA\to E_!FA.
\end{equation}
\end{blank}

\section{Deriving cartesian maps between indexed frames}

\begin{theorem}\label{thm:enough-fib-cofib}
Let $\trip:\catcpos$ be an indexed frame.
\begin{inumerate}
\item If $\trip$ is a tripos, then $\canp$ has enough \pfib objects.
\item If $\trip$ has enough $\exists$-prime predicates, then $\canp$ has enough
proto-cofibrant objects.
\end{inumerate}
\end{theorem}
\begin{proof}
For partial equivalence relations in triposes the property of being
proto-fibrant is equivalent to what is called \emph{weakly complete}
in~\cite[Definition~3.2]{hjp80}, and from the proofs of Lemma~3.1 and
Proposition~3.3 of loc.~cit.~one can extract a proof that $\canp$ has enough
\pfib objects. Specifically, a proto-fibrant replacement of $\aro$ is given by
\begin{equation}\nameof{\rho}:\aro\to(\tripower(A),\overline\rho)\end{equation} 
where $\nameof{\rho}$ is as in \ref{ax:po} and 
\begin{align}
  \overline\rho(m,n)\;\equiv
                  &\; (\fa x\qdot \ve_A(x,m)\Leftrightarrow \ve_A(x,n))
                \\&\wedge(\fa x \,y \qdot \ve_A(x,m)\wedge \ve_A(y,m)\imp \rho(x,y))
                \\&\wedge(\ex x\qdot \ve_A(x,m)).
\end{align}
The second claim follows from
Proposition~\ref{prop:enough-cofibs} and Examples~\ref{ex:proto-fib-cofib}.
\end{proof}
\begin{definition}
Let $\Phi:\trip\to\triq:\catcpos$ be a cartesian map between indexed frames. 
The functor 
\begin{equation}\label{eq:c-phi}
\catc\langle\Phi\rangle\;:\;\canp\;\to\;\canq
\end{equation}
sends objects $\aro\in\canp$ to objects $(A,\Phi_{A\times
A}(\rho))\in\canq$, and morphisms in $\canp$ to morphisms in $\canq$ having the
same underlying map in $\catc$. 
\end{definition}
For the statement of the next theorem  we use the terminology of `derived
functor', which we recall from \cite[I-2.3~(v)]{dwyer2005homotopy}.
\begin{definition}Let $\mcc$ and $\mcd$ be we-categories with localization
functors $E_\mcc:\mcc\to\ho(\mcc)$ and $E_\mcd:\mcd\to\ho(\mcd)$, and let 
$F:\mcc\to\mcd$ be an arbitrary functor.
\begin{anumerate}
\item A \emph{left derived functor} of $F$ is a \emph{right} Kan extension of 
$E_\mcd\circ F$ along $E_\mcc$.
\item A \emph{right derived functor} of $F$ is a \emph{left} Kan extension of 
$E_\mcd\circ F$ along $E_\mcc$.
\end{anumerate}
\end{definition}
\begin{theorem}\label{thm:derived}
 Let 
$
  \Phi:\trip\to\triq : \catcpos
$
be a cartesian map between indexed frames.
\begin{inumerate}
  \item The functor $\catc\langle\Phi\rangle$ preserves finite limits.
  \item We have $E_\triq(\catc\langle\Phi\rangle
  (f))=E_\triq(\catc\langle\Phi\rangle (g))$ for all parallel maps $f,g:A\to B$
  in $\canp$ with $E_\trip(f)=E_\trip(g)$,
  where $E_\trip$ an $E_\triq$ are the respective localization functors.
  \item\label{thm:derived-right}
  If $\trip$ is a tripos then $\catc\langle\Phi\rangle$ admits a right derived
  functor.
  \item\label{thm:derived-left}
  If $\trip$ has enough \eprime predicates then $\catc\langle\Phi\rangle$
  admits a left derived functor.
\end{inumerate}
\end{theorem}
\begin{proof}
The first two claims are straightforward (for the second one can use
Lemma~\ref{lem:homotopy}). The third and fourth claim follow from
Lemma~\ref{lem:qfib-kan} and its dual, respectively, whose hypotheses are
satisfied by the second claim and the two parts of
Theorem~\ref{thm:enough-fib-cofib}.
\end{proof}
\begin{notation}
In the situation of Theorem~\ref{thm:derived}, whenever they exist, we 
denote the left and right derived functors of $\canphi$ by 
$L_\Phi,R_\Phi:\csqp\to\csqq$, respectively.
\end{notation}
\begin{lemma}\label{lem:monic}
Let $\trip:\catcpos$ be a tripos with enough \eprime predicates and let
$\Phi:\trip\to\triq$ be a cartesian map into an indexed frame. Then the 
comparison natural transformation $\xi:L_\Phi\to R_\Phi$ is componentwise monic.
\end{lemma}
\begin{proof}
Let $\aro\in\canp$ and let 
\begin{equation}
(\underline{A},\underline{\rho})\xrightarrow{\kappa}(A,\rho)\xrightarrow{\iota}
(\overline{A},\overline{\rho})
\end{equation}
be proto-cofibrant and proto-fibrant replacements. Then $\iota\circ\kappa$
is a weak equivalence in $\canp$, in particular we have 
\begin{description}
\item[(inj)] $\underline{\rho}x,\underline{\rho}y,\overline{\rho}(\iota(\kappa
x),\iota(\kappa y))\ent_{x,y}\underline{\rho}(x,y)\,$.
\end{description}
The validity of this judgment is preserved by $\canphi$, which implies that 
the image under $E_\triq$ is monic in $\catc[\triq]$, as in the proof of
Theorem~\ref{thm:canp-hocat-catcp}.
\end{proof}
\begin{remark}
As is elaborated in~\cite{frey2015triposes} using different terminology, the 
universal property of right derived functors implies that the assignment 
$\Phi\mapsto R_\Phi$ is \emph{oplax functorial} in the sense that there is a 
canonical natural transformation $R_{\Psi\circ\Phi}\to R_{\Psi}\circ R_\Phi$
whenever the respective derived functors exist. Dually, the assignment 
$\Phi\mapsto L_\Phi$ is lax functorial.
\end{remark}
We conclude by giving some examples.
\begin{examples}\label{ex:derived}
\begin{anumerate}
\item If $\Phi:\trip\to\triq$ is a morphism of indexed frames then $\canphi$
preserves weak equivalences and therefore $E_\triq\circ\canq$ inverts them. By
the remarks in~\ref{bl:compare} we conclude that $L_\Phi$ and $R_\Phi$ can be
chosen to be equal and then the comparison $\xi$ is an identity.
\item As mentioned in the introduction, for cartesian maps $\Phi:\trip\to\triq$
between \emph{triposes} the functors $R_\Phi$ were used (using different
terminology) in~\cite[Section~3]{hjp80} to construct geometric morphisms between
toposes from adjunctions of cartesian indexed monotone maps between triposes.
\item For every $\catset$-based tripos $\trip:\setpos$ we can define a cartesian
transformation
\begin{equation}
\gamma : \trip\to\sub(\catset),\qquad\gamma_I
(\varphi)=\setof{i\in I}{\top\leq (1\xrightarrow{i}I)^*(\varphi)}
\end{equation}
\ie $\gamma_I(\vf)$ is the subset of $I$ where $\vf$ is `pointwise true'.
We have $\catset[\sub(\catset)]\simeq\catset$, and modulo this equivalence
$R_\gamma$ coincides with the global sections functor
\begin{equation}
\Gamma=\catset[\trip](1,-):\catset[\trip]\to\catset\,.
\end{equation}
If $\trip$ has enough \eprime predicates then $\canphi$ also admits a left
derived functor by Theorem~\ref{thm:derived}. In this case we have to
distinguish two cases.
\begin{itemize}
\item If $\top\in\trip(1)$ is \eprime then $\gamma$ is a morphism of indexed
frames and $L_\gamma$ and $R_\gamma$ coincide.
\item Otherwise, $L_\gamma$ is constant $0$ since
$\gamma$ sends all \eprime predicates to $\bot$.
\end{itemize}
\item Let $j:\sfop{eff}\to\sfop{eff}$ be the local operator 
on the effective tripos corresponding to Lifschitz
realizability~\cite[Section~4.4]{vanoosten2008realizability}. 
Its right derived functor $R_j:\cateff\to\cateff$ is the cartesian 
reflector corresponding to the Lifschitz-subtopos of the effective topos
$\cateff=\catset[\treff]$.

To understand $L_j$, let $\cta\in\catset\angs{\treff}$ be a partial equivalence
relation with \eprime support and fibrant replacement
$\iota:\cta\to(\overline{C},\overline{\tau})$.
The composite
\begin{equation}
\cta\to({C},j\,{\tau})\xrightarrow{\iota}(\overline{C},j\,\overline{\tau})
\end{equation}
in $\catset\angs{\treff}$ is mapped by $E:\catset\angs{\treff}\to\cateff$
to 
\begin{equation}\begin{tikzcd}[column sep = large]
    \cta
        \ar[dr,"\eta"', shift right = 1]
        \ar[r,"f"']
&   L_j\cta
        \ar[d,"\xi",tail]
\\  
&   R_j\cta
\end{tikzcd}\end{equation}
where $\eta$ is the unit of the reflector. We've shown in Lemma~\ref{lem:monic}
that $\xi$ is monic. Moreover, it can be seen relatively easily from the
definition in~\cite{vanoosten2008realizability} that $j(\varpi)=\varpi$ for any
\eprime predicate $\varpi$ in $\sfop{eff}$, which means by
Remark~\ref{rem:surj-reform} that the judgment \ref{ax:surj} holds for the map
$\cta\to({C},j\,{\tau})$. From this we can deduce using
Lemma~\ref{lem:frel-iso}\ref{lem:frel-iso-iso} that its image $f$  under the
localization functor is a cover, and we get a characterization of $L_j\cta$ as
fitting into the middle of the cover-mono factorization of $\cta\to R_j\cta$.

In other words, $L_j$ is the reflection onto the $j$-separated objects in
$\cateff$.
\item Consider the poset $P=\{l\leq\top\geq r\}$ and let $f:\low(P)\to\{0\leq
1\}$ be the unique map with $f\inv(\{1\})=\{\{l,r\},\{l,\top,r\}\}$. 
Then $f$ preserves finite meets. Moreover, we have 
\begin{equation}
\catset[\fam(\low(P))]\simeq\catset^{P\op}\qqtext{and}\catset[\fam(\{0\leq
1\})]\simeq\catset\,,
\end{equation}
and modulo these equivalences the left and right derived
functors of $\catset\angs{\fam(f)}$ are given by 
\begin{empheq}{align}
L_f,R_f\scol\catset^{P\op}&\to\catset\\
L_f(A\leftarrow C\to B) &= \image(C\to A\times B)\\
R_f(A\leftarrow C\to B) &= A\times B.
\end{empheq}
\end{anumerate}
\end{examples}

\providecommand{\bysame}{\leavevmode\hbox to3em{\hrulefill}\thinspace}
\providecommand{\MR}{\relax\ifhmode\unskip\space\fi MR }
\providecommand{\MRhref}[2]{%
  \href{http://www.ams.org/mathscinet-getitem?mr=#1}{#2}
}
\providecommand{\href}[2]{#2}

\bigskip\footnotesize
Jonas Frey, 
\textsc{
Department of Philosophy, 
Carnegie Mellon University, 
5000 Forbes Avenue, 
Pittsburgh, PA 15213, 
USA}
\par\nopagebreak
\textit{E-mail address}: \texttt{jonasf@andrew.cmu.edu}


\begin{thebibliography}{{Mac}98}

\bibitem[Bau89]{baues1989algebraic}
H.J. Baues, \emph{Algebraic homotopy}, Cambridge Studies in Advanced
  Mathematics, vol.~15, Cambridge University Press, Cambridge, 1989.

\bibitem[Bro73]{brown1973abstract}
K.S. Brown, \emph{Abstract homotopy theory and generalized sheaf cohomology},
  Transactions of the American Mathematical Society \textbf{186} (1973),
  419–458.

\bibitem[CV98]{carboni1998regular}
A.~Carboni and E.M. Vitale, \emph{Regular and exact completions}, J. Pure Appl.
  Algebra \textbf{125} (1998), no.~1-3, 79–116. \MR{1600009 (99b:18003)}

\bibitem[{Dav}79]{davey1979lattice}
B.A. {Davey}, \emph{On the lattice of subvarieties}, Houston Journal of
  Mathematics \textbf{5} (1979), 183--192 (English).

\bibitem[DHKS04]{dwyer2005homotopy}
W.G. Dwyer, P.S. Hirschhorn, D.M. Kan, and J.H. Smith, \emph{Homotopy limit
  functors on model categories and homotopical categories}, Mathematical
  Surveys and Monographs, vol. 113, American Mathematical Society, Providence,
  RI, 2004.

\bibitem[Fre13]{frey2013fibrational}
J.~Frey, \emph{A fibrational study of realizability toposes}, Ph.D. thesis,
  Paris 7 University, 2013.

\bibitem[Fre15]{frey2015triposes}
\bysame, \emph{Triposes, q-toposes and toposes}, Annals of Pure and Applied
  Logic \textbf{166} (2015), no.~2, 232–259.

\bibitem[GZ67]{gabriel1967calculus}
P.~Gabriel and M.~Zisman, \emph{Calculus of fractions and homotopy theory},
  Ergebnisse der Mathematik und ihrer Grenzgebiete, Band 35, Springer-Verlag
  New York, Inc., New York, 1967. \MR{0210125}

\bibitem[HJP80]{hjp80}
J.M.E. Hyland, P.T. Johnstone, and A.M. Pitts, \emph{Tripos theory}, Math.
  Proc. Cambridge Philos. Soc. \textbf{88} (1980), no.~2, 205–232.

\bibitem[Hof06]{hofstra2006all}
P.J.W. Hofstra, \emph{All realizability is relative}, Math. Proc. Cambridge
  Philos. Soc. \textbf{141} (2006), no.~02, 239–264.

\bibitem[Jac01]{jacobs2001categorical}
B.~Jacobs, \emph{Categorical logic and type theory}, Elsevier Science Ltd,
  2001.

\bibitem[Joh02a]{elephant1}
P.T. Johnstone, \emph{Sketches of an elephant: a topos theory compendium.
  {V}ol. 1}, Oxford Logic Guides, vol.~43, Oxford University Press, New York,
  2002. \MR{MR1953060 (2003k:18005)}

\bibitem[Joh02b]{elephant2}
\bysame, \emph{Sketches of an elephant: a topos theory compendium. {V}ol. 2},
  Oxford Logic Guides, vol.~44, Oxford University Press, Oxford, 2002.

\bibitem[KS17]{kapulkin2017quasicategories}
K.~Kapulkin and K.~Szumi{\l}o, \emph{Quasicategories of frames of cofibration
  categories}, Applied Categorical Structures \textbf{25} (2017), no.~3,
  323--347.

\bibitem[Law70]{lawvere1970equality}
F.W. Lawvere, \emph{Equality in hyperdoctrines and the comprehension schema as
  an adjoint functor}, Applications of Categorical Algebra \textbf{17} (1970),
  1–14.

\bibitem[Law07]{lawvere2007axiomatic}
\bysame, \emph{Axiomatic cohesion}, Theory and Applications of Categories
  \textbf{19} (2007), no.~3, 41--49.

\bibitem[{Mac}98]{maclanecwm}
S.~{Mac Lane}, \emph{Categories for the working mathematician}, second ed.,
  Graduate Texts in Mathematics, vol.~5, Springer-Verlag, New York, 1998.
  \MR{1712872 (2001j:18001)}

\bibitem[MPR17]{maietti2017triposes}
M.~Maietti, F.~Pasquali, and G.~Rosolini, \emph{Triposes, exact completions,
  and {H}ilbert's {$\varepsilon$}-operator}, Tbilisi Mathematical Journal
  \textbf{10} (2017), no.~3, 141--166. \MR{3725757}

\bibitem[MR12a]{maietti2012elementary}
M.~Maietti and G.~Rosolini, \emph{Elementary quotient completion}, Theory and
  Applications of Categories \textbf{27} (2012), Paper No. 17, 463.
  \MR{3699058}

\bibitem[MR12b]{maietti2012unifying}
\bysame, \emph{Unifying exact completions}, Applied Categorical Structures
  (2012), 1–10.

\bibitem[Pit81]{pitts81}
A.M. Pitts, \emph{The theory of triposes}, Ph.D. thesis, Cambridge Univ., 1981.

\bibitem[Pit02]{pitts2002}
\bysame, \emph{Tripos theory in retrospect}, Math. Structures Comput. Sci.
  \textbf{12} (2002), no.~3, 265–279.

\bibitem[Ros16]{rosolini2016category}
G.~Rosolini, \emph{The category of equilogical spaces and the effective topos
  as homotopical quotients}, Journal of Homotopy and Related Structures
  \textbf{11} (2016), no.~4, 943--956.

\bibitem[Ste13]{stekelenburg2013realizability}
W.P. Stekelenburg, \emph{Realizability categories}, Ph.D. thesis, Utrecht
  University, 2013.

\bibitem[Tro20]{trotta2020existential}
D.~Trotta, \emph{The existential completion}, Theory and Applications of
  Categories \textbf{35} (2020), Paper No. 43, 1576--1607. \MR{4162120}

\bibitem[vdB20]{vandenberg2020univalent}
B.~van~den Berg, \emph{Univalent polymorphism}, Annals of Pure and Applied
  Logic \textbf{171} (2020), no.~6, 102793, 29. \MR{4083000}

\bibitem[vdBM16]{vandenberg2016exact}
B.~van~den Berg and I.~Moerdijk, \emph{Exact completion of path categories and
  algebraic set theory}, arXiv preprint arXiv:1603.02456 (2016).

\bibitem[vO08]{vanoosten2008realizability}
J.~van Oosten, \emph{Realizability: {A}n {I}ntroduction to its {C}ategorical
  {S}ide}, Elsevier Science Ltd, 2008.

\bibitem[vO15]{vanoosten2015notion}
\bysame, \emph{A notion of homotopy for the effective topos}, Mathematical
  Structures in Computer Science \textbf{25} (2015), no.~5, 1132--1146.

\end{thebibliography}
\end{document}